\documentclass[10pt,a4paper,twoside,reqno]{amsart}

\usepackage[a4paper,
left=1in,right=1in,top=1in,bottom=1.2in,%
footskip=.25in]{geometry}

\usepackage{bm}
\usepackage{latexsym, amsmath, amstext, amssymb, amsfonts, amscd, bm, array, multirow, amsbsy, mathrsfs}
\usepackage{amsthm}
\usepackage{t1enc}
\usepackage[mathscr]{eucal}
\usepackage{indentfirst}
\usepackage{pb-diagram}
\usepackage{graphicx}
\usepackage{fancyhdr}
\usepackage{fancybox}
\usepackage{enumerate}
\usepackage{color}
\usepackage{tikz-cd}
\usepackage[all]{xy}
\usepackage{hyperref}
\usepackage{tikz}
\usepackage{xparse}
\hypersetup{colorlinks=false,pdfborderstyle={/S/U/W 0}}
\usetikzlibrary{matrix}

\usepackage{url}
\usepackage[sort&compress,comma]{natbib}
\bibpunct{[}{]}{,}{n}{}{,}
\theoremstyle{plain}
\newtheorem{thm}{Theorem}[section]

\newtheorem{prop}[thm]{Proposition}

\newtheorem{lemma}[thm]{Lemma}
\newtheorem{cor}[thm]{Corollary}

\theoremstyle{definition}
\newtheorem{definition}[thm]{Definition}

\theoremstyle{remark}
\newtheorem{remark}[thm]{Remark}
\newtheorem{example}[thm]{Example}

\newtheorem*{ack}{Acknowledgements}

\newcommand{\End}{\mathrm{End}}
\newcommand{\Aut}{\mathrm{Aut}}

\newcommand{\Mat}{\mathrm{Mat}}

\newcommand{\eqdef}{\stackrel{{\rm def.}}{=}}

\DeclareFontFamily{U}{rsf}{}
\DeclareFontShape{U}{rsf}{m}{n}{<5> <6> rsfs5 <7> <8> <9> rsfs7 <10-> rsfs10}{}
\DeclareMathAlphabet\Scr{U}{rsf}{m}{n}

\def\C{\mathbb{C}}
\def\R{\mathbb{R}}

\def\dd{\mathrm{d}}

\def\vol{\mathrm{vol}}

\def\Stab{\mathrm{Stab}}

\def\Diff{\mathrm{Diff}}
\def\Conf{\mathrm{Conf}}
\def\Sol{\mathrm{Sol}}

\newcommand{\be}{\begin{equation*}}
\newcommand{\ee}{\end{equation*}}
\newcommand{\ben}{\begin{equation}}
\newcommand{\een}{\end{equation}}
\newcommand{\beqa}{\begin{eqnarray*}}
	\newcommand{\eeqa}{\end{eqnarray*}}
\newcommand{\beqan}{\begin{eqnarray}}
\newcommand{\eeqan}{\end{eqnarray}}

\newcommand{\nn}{\nonumber}

\newcommand{\Tr}{\mathrm{Tr}}

\usepackage{float}
\restylefloat{table}

\def\cR{{\mathcal R}}

\def\cB{\Scr B}

\def\cH{\mathcal{H}}

\def\Cl{\mathrm{Cl}}

\def\Spin{\mathrm{Spin}}

\def\Spin{\mathrm{Spin}}

\def\SO{\mathrm{SO}}

\def\U{\mathrm{U}}

\def\cD{\mathcal{D}}
\def\cJ{\mathcal{J}}

\def\cE{\mathcal{E}}
\def\cI{\mathcal{I}}

\def\cN{\mathcal{N}}
\def\cG{\mathcal{G}}
\def\cT{\mathcal{T}}
\def\cF{\mathcal{F}}

\def\SU{\mathrm{SU}}
\def\Sp{\mathrm{Sp}}
\def\G_2{\mathrm{G_2}}

\def\cL{\mathcal{L}}
\def\cS{\mathcal{S}}

\def\cV{\mathcal{V}}

\newcommand{\Hom}{{\rm Hom}}

\newcommand{\Iso}{{\rm Iso}}

\def\Aut{\mathrm{Aut}}

\def\Re{\mathrm{Re}}
\def\Im{\mathrm{Im}}

\def\G{\mathrm{G}}

\def\R{\mathbb{R}}

\def\cM{\mathcal{M}}

\def\cL{\mathcal{L}}
\def\cH{\mathcal{H}}
\def\dd{\mathrm{d}}

\usepackage{enumitem}
\setlist[itemize]{leftmargin=*}

\newcolumntype{P}[1]{>{\centering\arraybackslash}p{#1}}

\usepackage[math]{anttor}
\usepackage[T1]{fontenc}
 
\begin{document}

\title[Geometric supergravity and electromagnetic duality]{Four-dimensional geometric supergravity and electromagnetic duality: a brief guide for mathematicians}


\author[C. I. Lazaroiu]{C. I. Lazaroiu} \address{Department of Theoretical Physics, Horia Hulubei National Institute for Physics and Nuclear Engineering, Bucharest, Romania}
\email{lcalin@theory.nipne.ro}

\author[C. S. Shahbazi]{C. S. Shahbazi} \address{Department of Mathematics, University of Hamburg, Germany}
\email{carlos.shahbazi@uni-hamburg.de}

\thanks{2010 MSC. Primary: 53C27. Secondary: 53C50.}
\keywords{Electromagnetic duality, supergravity, symplectic vector bundles, character varieties}

\begin{abstract} 
We give a gentle introduction to the global geometric formulation of the bosonic sector of four-dimensional supergravity on an oriented four-manifold $M$ of arbitrary topology, providing a geometric characterization of its U-duality group. The geometric formulation of four-dimensional supergravity is based on a choice of a vertically Riemannian submersion $\pi$ over $M$ equipped with a flat Ehresmann connection, which determines the non-linear \emph{section} sigma model of the theory, and a choice of flat symplectic vector bundle $\cS$ equipped with a positive complex polarization over the total space of $\pi$, which encodes the inverse gauge couplings and theta angles of the theory and determines its gauge sector. The classical fields of the theory consist of Lorentzian metrics on $M$, global sections of $\pi$ and two-forms valued in $\cS$ that satisfy an algebraic relation which defines the notion of \emph{twisted} self-duality in four Lorentzian dimensions. We use this geometric formulation to investigate the group of \emph{electromagnetic duality transformations} of supergravity, also known as the continuous classical U-duality group, which we characterize using a certain short exact sequence of automorphism groups of vector bundles. Moreover, we discuss the general structure of the Killing spinor equations of four-dimensional supergravity, providing several explicit examples and remarking on a few open mathematical problems. This presentation is aimed at mathematicians working in differential geometry.
\end{abstract}
 


\maketitle

\setcounter{tocdepth}{1} 
\tableofcontents


\section{Introduction}


Supergravity theories are supersymmetric theories of gravity which,
aside from their intrinsic phenomenological interest, are of fundamental importance in
high energy physics, since they describe the low-energy limit of string
and M-theory and their associated supersymmetric compactifications
\cite{BeckerBecker,Castellani:1991et,FreedmanProeyen,Grana:2005jc,Ortin}. In
addition to their role in theoretical physics, supergravity theories
have been the source of important developments and activity in
mathematics, especially in geometry and topology, see for instance
\cite{Freed,Friedrich:2001nh,Garcia-Fernandez:2016azr,Hitchin,Moore,YauProp}
as well as their references and citations. Indeed, the local formulation of
supergravity is well-known to involve important mathematical
structures and objects interacting in a delicate equilibrium dictated
by supersymmetry, such as K\"ahler-Hodge manifolds, Riemannian
manifolds of special holonomy, harmonic maps, exceptional Lie groups,
gerbes and Courant algebroids or gauge-theoretic moduli spaces, to
name only a few. This makes the mathematical study of supergravity
into a rich and quite formidable endeavour. Supergravity theories can
be defined in various dimensions and signatures and can be deformed
through various mechanisms while preserving their supersymmetric
structure (see \cite{Trigiante:2016mnt,Gallerati:2016oyo} and
references therein for details on the deformation of such theories
through \emph{gauging}). In this short review, we will consider
exclusively four-dimensional \emph{ungauged} supergravity theories in
Lorentzian signature, where the term ungauged indicates that we will not
consider any gauging of the theory. Such theories are particularly
relevant for several reasons, both from the physical and mathematical
point of view, among which we can mention the following:

\begin{itemize}
\item Four-dimensional supergravity theories describe the effective
dynamics of the massless modes of string and M-theory
compactifications to four-dimensions, which is the observed physical
dimension of spacetime and therefore yields the adequate set-up for
phenomenological applications \cite{BeckerBecker,Grimm:2005fa}.
	
\item Supergravity theories in four-dimensions enjoy a type of duality
called \emph{electromagnetic U-duality}, which is inherited from
ordinary electromagnetic duality in four Lorentzian dimensions and
has deep connections with string theory U-duality groups
\cite{Hull:1994ys}. Furthermore, electromagnetic duality gives rise to
interesting mathematical structures of gauge-theoretic type \cite{LazaroiuShahbazi}.

\item Four-dimensional supergravity theories involve rich non-linear
sigma models with Riemannian target spaces of special type \cite{FreedmanProeyen,Ortin,Huebscher:2006mr}, whose
moduli spaces of solutions can be expected to enjoy interesting
applications in the differential topology of Riemannian three and four
manifolds. 

\item The dimensional reduction of $\cN=2$ four-dimensional
supergravity to three-dimensions is the origin of the celebrated c-map
in quaternionic-K\"ahler and projective special K\"ahler geometry
\cite{FerraraSabharwal,deWit:2001brd,Hitchin}, see also
\cite{WitProeyen} for a related construction called the \emph{the
r-map} in the literature. In particular, $\cN=2$ supergravity has a deep
connection with Quaternionic-K\"ahler manifolds whose mathematical investigation has been already initiated in several pioneering works, see \cite{Alekseevsky:2013nua,Alekseevsky:2012fu,Cortes:2011aj,Hitchin,Macia:2014oca} and their citations.  
\end{itemize}

\noindent
In contrast to higher-dimensional supergravities, which receive
increasing mathematical attention and whose geometric formulation and
structure is being actively investigated
\cite{Babalic:2016mbw,Baraglia,Braunack-Mayer:2018uyy,Fiorenza:2017jqx,Garcia-Fernandez:2016azr,Grady:2019man,Sati:2010ss},
the mathematics community has paid little attention to low-dimensional
supergravity, in particular four-dimensional supergravity. This might be due in part to the
inaccessibility of the relevant physics literature to
mathematicians. Nonetheless, four-dimensional supergravity is an
extremely rich subject from a mathematical standpoint, and there
exists indeed a \emph{plethora} of such theories involving interesting modern mathematical structures and leading to novel mathematical problems, most of which have not emerged into the mathematical community. The distinction
between higher and low dimensional supergravity is akin to that
between higher and low-dimensional differential topology, the latter
yielding a remarkably rich and subtle theory \cite{Donaldson}. Given
their importance, the local structure and properties of four
dimensional supergravities have been extensively studied in the physics
literature, in a long-term effort that evolved myriad of
ramifications. We refer the reader to
\cite{Andrianopoli:1996cm,Andrianopoli:1996ve,Aschieri:2008ns,Ceresole:1995ca,Ceresole:1995jg,Cremmer:1982en,Cremmer:1982wb,deWit:1984rvr,deWit:1984wbb,Fre:1995dw,LopesCardoso:2019mlj} for more details and references. Despite all this work, the global
geometric formulation and proper mathematical theory of
four-dimensional supergravity are poorly understood and remain open
for investigation and exploration.

It would be desirable to develop the complete mathematical foundations
of all four-dimensional supergravities, including their bosonic and
fermionic sectors. In our opinion, this may be currently out of
reach. Fortunately, most applications of supergravity to differential
geometry and topology only require the mathematical theory of the
bosonic sector together and the Killing spinor equations, which fully
capture the geometry and topology of supersymmetric solutions and
associated moduli spaces. Solutions of the equations of motion of a
supergravity theory which satisfy the supergravity Killing spinor
equations are called \emph{supersymmetric solutions} and have been
intensively studied in the physics \cite{Gran:2018ijr} and mathematics
literature \cite{Friedrich:2001nh}, the latter focusing almost
entirely on higher-dimensional Riemannian signature. The global
geometrization of bosonic supergravity together with its associated
Killing spinor equations on oriented manifolds of arbitrary topology
was named \emph{geometric supergravity} in
\cite{Lazaroiu:2016iav,Lazaroiu:2017qyr,Cortes:2018lan}, which
initiated a long-term program devoted to systematically developing the
mathematical foundations of four-dimensional (ungauged) geometric
supergravity. The first step in this program concerns the bosonic
sector, paying special attention to its Dirac quantization,
electromagnetic U-duality group and various reductions to
three-dimensional Riemannian manifolds and Riemann surfaces. We note
that the mathematical theory of geometric supergravity is far from
finished, and \cite{Lazaroiu:2016iav,Lazaroiu:2017qyr,Cortes:2018lan}
constitute only a first few steps towards its completion.

In this short review we will discuss some of the results of
\cite{Lazaroiu:2016iav,Lazaroiu:2017qyr,Cortes:2018lan} concerning the
global mathematical formulation and symplectic duality structure of
geometric supergravity. Roughly speaking, the generic bosonic sector
of four-dimensional supergravity consists of three sub-sectors, namely:

\begin{itemize}
	\item The gravitational sector, which corresponds to the
Einstein-Hilbert term of the local Lagrangian.
	
	\item The scalar sector, which corresponds locally to a non-linear
sigma model coupled to gravity.
	
	\item The gauge sector, which corresponds locally to a theory
of an arbitrary number of \emph{abelian gauge fields} coupled to the
scalars fields of the scalar sector.
\end{itemize}

\noindent Therefore, four-dimensional supergravity can be though of as
the unification, using supersymmetry as a guiding principle, of three
cornerstones of differential geometry, namely the theory of Einstein
metrics, the theory of harmonic maps and Yang-Mills theory.

The Killing spinor equations of four-dimensional supergravity are
first order differential equations involving the bosonic fields of the
theory and a \emph{supersymmetry parameter}, which is mathematically
described as a section of an appropriate bundle of Clifford modules
over the underlying Lorentzian manifold. Supergravity Killing spinor
equations generalize, through the principle of
supersymmetry, well-known spinorial equations studied intensively in
the literature, such as Hermite-Yang-Mills equations, instanton
equations, the Seiberg-Witten equations, generalized Killing spinor
equations or the pseudoholomorphicity equations. The study of
supergravity Killing spinor equations makes contact with modern areas
of mathematics under current development and brings supergravity into
\emph{mathematical gauge theory}, a field of mathematics whose tools
and methods are specially well-adapted to the study of supersymmetric
solutions and their moduli spaces. We hope that the development of the
mathematical theory of four-dimensional supergravity can clarify this
relation and bring new problems and perspectives into mathematical gauge theory.

An important remark is in order: we do not discuss the Dirac
quantization of geometric supergravity in this report, since it is yet
to be fully developed and it is work in progress
\cite{LazaroiuShahbazi}. As shown in Op. Cit., implementing Dirac
quantization is a fundamental step in order to properly understand the
geometric structure of four-dimensional supergravity as well as the
global structure of its solutions and associated moduli spaces.

The outline of this manuscript is as follows. In Section
\ref{sec:localsugra} we review the well-known local formulation of
four-dimensional bosonic supergravity, giving a rigorous seemingly novel
description of its electromagnetic U-duality group. In Section
\ref{sec:geometricsugra} we explain the global geometric formulation
of bosonic four-dimensional supergravity together with the necessary
geometric background. In Section \ref{sec:globalautgroup} we describe
the global electromagnetic U-duality group of geometric supergravity,
characterizing it in terms of a certain short exact sequence and
discussing some examples. Finally, in Section \ref{sec:KSE} we briefly
discuss the Killing spinor equations of four-dimensional supergravity
and present some explicit examples, mentioning along the way some open
mathematical problems.


\section{Local bosonic supergravity}
\label{sec:localsugra}


In this section we review the local formulation of the generic bosonic
sector of four-dimensional supergravity, paying special attention to the electromagnetic
\emph{U-duality} group of the theory, which consists of 
electromagnetic duality transformations of the abelian gauge fields
coupled to scalars and gravity. The local formulation of the bosonic
sector of four-dimensional supergravity was considered in detail in
references
\cite{Andrianopoli:1996cm,Andrianopoli:1996ve,Andrianopoli:1996wf,Gaillard:1981rj},
where the duality transformations of the local theory were
investigated. The reader is referred to
\cite{Aschieri:2008ns,Ceresole:1995ca,Fre:1995dw,LopesCardoso:2019mlj}
for comprehensive reviews and exhaustive lists of references.

Let $U$ be a contractible non-empty oriented and relatively compact open subset 
of $\mathbb{R}^4$ with coordinates $\left\{ x^a\right\}$, where $a =1,
\hdots , 4$. Fixing non-negative integers $n_s,n_v$, the configuration
space of the local bosonic sector of extended
four-dimensional supergravity with $n_s$ scalar fields and $n_v$
abelian gauge fields is defined as the set of triples $(g, \phi, A)$
consisting of:

\

\begin{itemize}
\item A Lorentzian metric $g$ defined on $U$.

\

\item An $\mathbb{R}^{n_s}\,$-valued function $\phi\colon U\to
\mathbb{R}^{n_s}$ defined on $U$. We denote the components of $\phi$ by
$\phi^i\colon U \to \mathbb{R}$, with $i = 1, \hdots , n_{s}$ and
fix an oriented open subset $V\subset \mathbb{R}^{n_s}$
containing $\phi(U)$. The real functions $\left\{
\phi^i\right\}$ are the (locally-defined)  \emph{scalar fields} of the theory. We will refer to such functions $\phi\colon U\to \mathbb{R}^{n_s}$ as \emph{scalar maps}.

\

\item An $\mathbb{R}^{n_v}\,$-valued one-form $A\in
\Omega^1(U,\mathbb{R}^{n_v})$. When necessary, we will denote the components
of $A$ by $A^{\Lambda} \in \Omega^1(U)$, with $\Lambda = 1 , \hdots ,
n_v$, which correspond to the local $\U(1)$ \emph{gauge fields} of the
theory. We denote by:
\begin{equation*} F \eqdef \dd A \in \Omega^2(U,\mathbb{R}^{n_v})\, ,
\end{equation*}

\noindent the \emph{field strength} associated to $A$, whose
components will be denoted by $F^{\Lambda} = \dd A^{\Lambda} \in
\Omega^2(U)$.
\end{itemize}

\

\noindent
The {\em local} bosonic sector of extended four-dimensional supergravity is defined through the following action functional:
\begin{equation}
\label{eq:localsugra} \mathrm{S}_l[g_U,\phi, A] \eqdef \int_U\left\{-
\mathrm{R}_{g_U} + \cG_{ij}(\phi) \partial_a \phi^i\partial^a \phi^j +
\mathcal{R}_{\Lambda \Sigma}(\phi) F^{\Lambda}_{ab}\ast F^{\Sigma\,
ab} + \mathcal{I}_{\Lambda \Sigma}(\phi) F^{\Lambda}_{ab} F^{\Sigma\,
ab}\right\} \nu_{g_U} \, ,
\end{equation}
 
\noindent where:

\

\begin{itemize}
	\item $\nu_{g_U}$ is the Lorentzian volume form associated to
$g$ and the given orientation on $U$.
	
	\
	
	\item $\cG \in \Gamma(T^{\ast}V\odot T^{\ast}V)$ is a
Riemannian metric on $V$. We denote by:
	\begin{equation*} \cG(\phi) \eqdef \cG\circ \phi \colon U \to
\mathrm{Sym}(n_s,\mathbb{R})\, ,
	\end{equation*}
	
	\noindent the composition of $\cG$ with $\phi$ and by
$\cG_{ij}(\phi)$ the components of $\cG(\phi)$ in the Cartesian
coordinates of $V\subset \mathbb{R}^{n_s}$.
	
	\
	 
	\item $\mathcal{R},\mathcal{I}\colon V\to
\mathrm{Sym}(n_v,\mathbb{R})$ are smooth functions on $V$ valued in
the vector space of $n_v\times n_v$ square symmetric matrices with
real entries. We denote by:
	\begin{equation*} \cR(\phi) \eqdef \cR\circ \phi\colon U \to
\mathrm{Sym}(n_v,\mathbb{R})\, , \qquad \cI(\phi) \eqdef \cI\circ
\phi\colon U \to \mathrm{Sym}(n_v,\mathbb{R})\, ,
	\end{equation*}
	
	\noindent the compositions of $\cR$ and $\cI$ with $\phi$ and
        by $\mathcal{I}_{\Lambda \Sigma}(\phi)$, $\mathcal{R}_{\Lambda
\Sigma}(\phi)$ the entries of the corresponding symmetric matrices. Furthermore, $\cI$
is required to be positive definite, a condition which is imposed in
order to have a consistent kinetic term for the gauge fields
$A^{\Lambda}$.
\end{itemize}

\

\noindent
Therefore, the local bosonic sector of supergravity on the oriented
open sets $(U,V)$ is uniquely determined by a choice of Riemannian
metric $\cG$ on $V$ and matrix-valued functions $\cI$ and $\cR$ as
described above. In some cases a scalar potential can occur in
\eqref{eq:localsugra}, but we have set it to zero for simplicity. The
functional $\mathrm{S}_l$ can be naturally written as a sum of three
pieces:
\begin{equation*}
\mathrm{S}_l = \mathrm{S}^{e}_l + \mathrm{S}^{s}_l + \mathrm{S}^{v}_l\, ,  
\end{equation*}

\noindent
where: 
\begin{equation*}
\mathrm{S}^e_l[g] \eqdef - \int_U \mathrm{R}_{g} \nu_{g} \, ,
\end{equation*}

\noindent
is the Einstein-Hilbert action on $U$,
\begin{equation*}
\mathrm{S}^s_l[g,\phi] \eqdef \int_U\,\cG_{ij}(\phi) \partial_a \phi^i\partial^a \phi^j\, \nu_{g} \, ,
\end{equation*}

\noindent
is a local non-linear sigma model with target space metric $\cG$, and
\begin{equation*}
\mathrm{S}^v_l[g,\phi, A] \eqdef \int_U\left\{ \mathcal{R}_{\Lambda \Sigma}(\phi) F^{\Lambda}_{ab}\ast F^{\Sigma\, ab} + \mathcal{I}_{\Lambda \Sigma}(\phi) F^{\Lambda}_{ab} F^{\Sigma\, ab}\right\} \nu_{g} \, ,
\end{equation*}

\noindent is a local Abelian Yang-Mills theory coupled to the scalars
$\left\{\phi^i\right\}_{i = 1, \hdots , n_s}$.

\begin{remark} In standard supergravity terminology,
$\mathrm{S}^{e}_l$ defines the \emph{gravity sector} of the theory,
$\mathrm{S}^{s}_l$ defines the \emph{scalar sector} of the theory and
$\mathrm{S}^{v}_l$ defines the \emph{gauge sector} of the theory.
\end{remark}

\noindent The matrix $\cI$ generalizes the inverse of the squared
coupling constant appearing in ordinary four-dimensional gauge
theories, whereas $\cR$ generalizes the \emph{theta angle} of quantum
chromodynamics. All together, the generic bosonic sector of extended
supergravity couples Einstein-Hilbert's action to a non-linear sigma
model with Riemannian target space $(V,\cG)$ and to a given number of
\emph{abelian gauge fields}. In supergravity terminology, the
Riemannian manifold $(V,\cG)$ is called the \emph{scalar manifold} of
the theory and $\cG$ its \emph{scalar metric}.

\begin{definition} We define a {\bf local electromagnetic structure}
on $V$ to be a pair $(\cR,\cI)$, where both $\cR$ and $\cI$ are symmetric
$n_v\times n _v$ matrix-valued functions on $V$ with $\cI$
positive-definite. We will denote by $\mathfrak{E}_V$ the set of all
electromagnetic structures on $V$.  We define a {\bf local
scalar-electromagnetic structure} on $V$ to be a triple
$(\cG,\cR,\cI)$, where $\cG$ is a Riemannian metric on $V$ and
$(\cR,\cI)$ is an electromagnetic structure. We will refer to the
local supergravity with scalar metric $\cG$ and gauge couplings
$(\cR,\cI)$ simply as the local supergravity associated to
$(\cG,\cR,\cI)$.
\end{definition}

\noindent Supersymmetry constrains the local isometry type of the
Riemannian manifold $(V,\cG)$ that can be considered as the target
space of the non-linear sigma model of a given supergravity
theory. Depending on the amount $\cN$ of supersymmetry preserved, the
local isometry type of $(V,\cG)$ is given as follows \cite{Andrianopoli:1996ve}:

\begin{table}[H]
	\centering
	\begin{tabular}{|P{5cm}|P{7cm}|P{2.5cm}|}
		\hline
		Number of supersymmetries & Isometry type of $(V,\cG)$ & Dimension \\ \hline
		$\cN= 1$ & $\cM_{\mathrm{KH}}$ & $2n_c$  \\ \hline
		$\cN= 2$ & $\cM_{\mathrm{PSK}}\times \cM_{\mathrm{QK}}$ & $2n_v + 4n_H$  \\ \hline
		$\cN= 3$ & $\mathrm{SU}(3,n)/\mathrm{S}(\U(3)\times \U(n))$ & $6n_v$  \\ \hline
		$\cN= 4$ & $\mathrm{SU}(1,1)/\U(1)\times\mathrm{SO}(6,n)/\mathrm{S}(\mathrm{O}(6)\times \mathrm{O}(n))$ & $6n_v + 2$  \\ \hline
		$\cN= 5$ & $\mathrm{SU}(1,5)/\mathrm{S}(\U(1)\times \U(5))$ & 10  \\ \hline
		$\cN= 6$ & $\mathrm{SO}^{\ast}(12)/(\U(1)\times \SU(6))$ & 30  \\ \hline
		$\cN= 8$ & $\mathrm{E}_{7(7)}/(\mathrm{S}\U(8)/\mathbb{Z}_2)$ & 70  \\ 
		\hline
	\end{tabular}
  \newline\newline
\caption{Isometry type of the scalar manifolds of four-dimensional
supergravity, depending on the amount $\cN$ of supersymmetry of the
theory. The symbol $n_c$ denotes the number of chiral multiplets,
$n_v$ denotes the number of vector multiplets and $n_H$ denotes the
number of hypermultiplets.}\label{tableSusyScalar}
\end{table}

\begin{remark} The case $\cN=7$ does not appear in the previous list
because $\cN=7$ supergravity can be shown to always admit an additional
supersymmetry which automatically makes it into $\cN=8$ supergravity \cite{Castellani:1991et}.
\end{remark}

\noindent The symbol $\cM_{\mathrm{KH}}$ denotes a K\"ahler-Hodge
manifold, or more precisely a complex manifold equipped with a chiral
triple \cite{Cortes:2018lan}, whereas $\cM_{\mathrm{PSK}}$ and
$\cM_{\mathrm{QK}}$ respectively denote a projective special K\"ahler
manifold and a Quaternionic-K\"ahler manifold. For $\cN>2$, the scalar
manifolds appearing in the previous table are all simply connected and
non-compact symmetric manifolds equipped with a certain Riemannian
metric. All of them are diffeomorphic to $\mathbb{R}^k$ for an
appropriate $k$. In the $\cN= 8$ case, $\mathrm{E}_{7(7)}$ denotes the
maximally non-compact real form of the complex exceptional Lie group
$\mathrm{E}_7$ and $\mathrm{S}\U(8)/\mathbb{Z}_2\subset \mathrm{E}_{7(7)}$
is its maximal compact subgroup.


\subsection{Equations of motion}
\label{sec:eqsmotion}


Let $\mathrm{G}^g_{ab} \eqdef \mathrm{R}^g_{ab} - \frac{1}{2}g_{ab} R^g$ denote the Einstein tensor associated to $g$. The equations of motion that follow from the action functional \eqref{eq:localsugra} for a given local scalar-electromagnetic structure $(\cG,\cR,\cI)$ are the following:

\

\begin{itemize}
	\item The \emph{Einstein equations}:
\begin{equation}
\label{eq:Einsteinlocal}
\mathrm{G}^g_{ab}= \cG_{ij}(\phi) \partial_a \phi^i \partial_b \phi^j - \frac{1}{2} g_{ab} \cG_{ij}(\phi) \partial_c \phi^i \partial^c \phi^j + 2 \cI_{\Lambda \Sigma}(\phi) F^{\Lambda}_{ac} F^{\Sigma c}_b - \frac{1}{2} g_{ab} \cI_{\Lambda \Sigma}(\phi)  F^{\Lambda}_{cd} F^{\Sigma cd}\, .
\end{equation}

\item The \emph{scalar equations}:
\begin{equation}
\label{eq:Scalarlocal}
\nabla_a^g (\cG_{ik}(\phi)  \partial^a \phi^i) = \frac{1}{2} \partial_k \cG_{ij}(\phi) \partial_a \phi^i \partial^a \phi^j + \frac{1}{2} \partial_k \cR_{\Lambda \Sigma}(\phi) F^{\Lambda}_{ab} \ast F^{\Sigma ab} + \frac{1}{2} \partial_k \cI_{\Lambda \Sigma}(\phi) F^{\Lambda}_{ab} F^{\Sigma ab}\, .
\end{equation}

\item The \emph{Maxwell equations}:
\begin{equation}
\label{eq:Maxwelllocal}
\nabla_a^g (\cR_{\Lambda \Sigma}(\phi) \ast F^{\Sigma ab} + \cI_{\Lambda \Sigma}(\phi) F^{\Sigma ab}) = 0\, .
\end{equation}

\end{itemize}

\

\noindent
The variables of the supergravity equations consist on Lorentzian metrics $g$ on $U$, $n_s$ scalars $\left\{\phi^i\right\}$ and $n_v$ closed two-forms $\left\{ F^{\Lambda} \right\}$. Conditions $\dd F^{\Lambda} = 0$, $\Lambda = 1,\hdots , n_v$, are known as the \emph{Bianchi identities}, and ensure that $F = \dd A$ for a vector valued one-form $A$ on $U$ . It can be easily seen that the Maxwell equations are equivalent to:
\begin{equation*}
\dd (\cR_{\Lambda \Sigma}(\phi) F^{\Sigma}) = \dd (\cI_{\Lambda \Sigma}(\phi) \ast F^{\Sigma})\, .
\end{equation*}

\noindent
Define now the two-forms:
\begin{equation*}
G_{\Lambda}(\phi) \eqdef \cR_{\Lambda \Sigma}(\phi) F^{\Sigma} - \cI_{\Lambda \Sigma}(\phi) \ast F^{\Sigma} \in \Omega^2(U)\, , \qquad \Lambda = 1,\hdots , n\, .
\end{equation*}

\noindent For ease of notation, we will sometimes drop the explicit
dependence of $G_{\Lambda}(\phi)$ on $\phi$. Then, the Bianchi
identities and Maxwell equations \eqref{eq:Maxwelllocal} are given by:
\begin{equation*}
\dd F^{\Sigma} = 0\, , \qquad \dd G_{\Lambda} = 0\, , \qquad \Lambda = 1,\hdots , n
\end{equation*}

\noindent
which in turn can be equivalently written simply as:
\begin{equation*}
\dd\cV(\phi) = 0\, ,
\end{equation*}

\noindent
where $\cV(\phi)\in \Omega^2(U,\mathbb{R}^{2n})$ denotes the following vector of two-forms:
\begin{equation*}
\cV(\phi) 
=   
\begin{pmatrix} 
F  \\
G(\phi)  
\end{pmatrix}\in \Omega^2(U,\mathbb{R}^{2n_v}) \, .
\end{equation*}

\noindent
The following important lemma follows by direct computation.

\begin{lemma}
\label{lemma:twistedselfdual}
Let $(\cR,\cI)$ be a local electromagnetic structure. A vector-valued two-form $\cV\in \Omega^2(U,\mathbb{R}^{2n_v})$ can be written as:
\begin{equation*}
\cV
=   
\begin{pmatrix} 
F  \\
G(\phi)  
\end{pmatrix} \, ,
\end{equation*}

\noindent
for $F\in \Omega^2(U,\mathbb{R}^{n_v})$, where $\phi\colon U\to \mathbb{R}^{n_s}$ and $G(\phi) = \cR(\phi) F - \cI(\phi) \ast F$, if and only if:
\begin{equation}
\label{eq:twistedselfdualitylocal}
\ast\cV = - \cJ(\phi)(\cV)\, , 
\end{equation}

\noindent
where $\cJ\colon V\to \mathrm{Gl}(2n_v,\mathbb{R})$ is the matrix-valued map defined as follows
\begin{equation*}
\cJ=  
\begin{pmatrix} 
 - \cI^{-1} \cR & \cI^{-1} \\
-\cI  - \cR\cI^{-1}\cR & \cR \cI^{-1} 
\end{pmatrix} \colon V\to \mathrm{Gl}(2n_v,\mathbb{R}) \, ,
\end{equation*}

\noindent
and $\cJ(\phi) \eqdef \cJ\circ\phi\colon U \to \mathrm{Gl}(2n_v,\mathbb{R})$. In particular, we have $\cJ^2 = -1$.
\end{lemma}

\begin{remark} The matrix-valued map $\cJ\colon V\to
\mathrm{Gl}(2n_v,\mathbb{R})$ can be understood as a fiber-wise complex
structure on the trivial vector bundle of rank $2n_v$ over $V$.
\end{remark}

\noindent
Equation \eqref{eq:twistedselfdualitylocal} is known in the literature
as the \emph{twisted self-duality condition} for the field strength
$\cV\in \Omega^2(U,\mathbb{R}^{2n_v})$. The following proposition
gives the geometric interpretation of condition
\eqref{eq:twistedselfdualitylocal}, which in turn unveils the global geometric
interpretation of the twisted self-duality condition, as we will see
in Section \ref{sec:geometricsugra}. For future reference, we define
the \emph{standard symplectic form} of $\mathbb{R}^{2n}$ to have the
matrix representation:
\begin{equation}
\omega =  
\begin{pmatrix} 
0 & -\mathrm{Id}\\
\mathrm{Id} & 0
\end{pmatrix}  
\end{equation}

\noindent
in the canonical basis of $\mathbb{R}^{2n_v}$. More precisely, if we denote the
canonical basis of $\mathbb{R}^{2n_v}$ by $\cE = (e_1 , \hdots ,
e_{n_v} , f_1,\hdots , f_{n_v})$, we have:
\begin{equation}
\label{eq:omegastandard}
\omega \eqdef \sum_a f_a^{\ast} \wedge e_a^{\ast}\, , \qquad a = 1, \hdots , n_v\, ,
\end{equation}

\noindent
where $\cE^{\ast} = (e_1^{\ast} , \hdots , e^{\ast}_{n_v} ,
f^{\ast}_1,\hdots , f^{\ast}_{n_v})$ is the basis dual to $\cE = (e_1
, \hdots , e_{n_v} , f_1,\hdots , f_{n_v})$.

\begin{prop}
\label{prop:cNTaming}
Let $\omega$ be the standard symplectic form on $\mathbb{R}^{2n}$. A matrix-valued map $\cJ\colon V \to \Aut(\mathbb{R}^{2n})$ can be written as:
\begin{equation}
\label{eq:Jlocaltaming}
\cJ= 
\begin{pmatrix} 
- \cI^{-1} \cR & \cI^{-1} \\
- \cI - \cR\cI^{-1}\cR & \cR \cI^{-1} 
\end{pmatrix}  \colon V \to \Aut(\mathbb{R}^{2n})
\end{equation}

\noindent
for a local electromagnetic structure $(\cR,\cI)$ if and only if $\cJ\vert_p$ is a compatible taming of $\omega$ for every $p\in V$. 
\end{prop}

\begin{remark}
  We recall that a complex structure $J$ on $\mathbb{R}^{2n_v}$ is said to be a {\em compatible taming} of $\omega$ if:
\begin{equation*}
\omega(J \xi_1, J \xi_2 ) = \omega(\xi_1,\xi_2)\, , \qquad \forall\,\, \xi_1 , \xi_2 \in \mathbb{R}^{2n_v}\, ,
\end{equation*}

\noindent
and:
\begin{equation*}
\omega(\xi,J\xi) > 0\, , \qquad \forall\,\, \xi\in \mathbb{R}^{2n_v}\backslash\left\{0\right\}\, .
\end{equation*}

\noindent In the following we shall always consider $\mathbb{R}^{2n_v}$
to be endowed with the symplectic form $\omega$ as introduced above.
\end{remark}

\begin{proof} If $\cJ$ is taken as in equation \eqref{eq:Jlocaltaming}
for a certain local electromagnetic structure $(\cR,\cI)$ then a
direct computation shows that it is a compatible taming of
$\omega$. Conversely, assume that $\cJ$ is a complex structure
on $\mathbb{R}^{2n_v}$ taming $\omega$ at every point
in $V$ (we omit the evaluation at a point for ease of notation). Let
$\cE = (e_1 , \hdots , e_{n_v} , f_1,\hdots , f_{n_v})$ the canonical
basis of $\mathbb{R}^{2n}$, which is symplectic with respect to
$\omega$. The vectors $\cE_f = ( f_1,\hdots , f_{n_v})$ form a basis
of the complex vector space $(\mathbb{R}^{2n},\cJ)$, and there exists
a unique map $\cN\colon V\to \Mat(n,\mathbb{C})$ valued in the complex
$n\times n$ square matrices and satisfying:
\begin{equation}
\label{eq:localperiod}
e_{a} = \sum_{b} \cN_{ab}\, f_{b}\, , \qquad a , b = 1, \hdots , n_v\, .
\end{equation}

\noindent
Then:
\begin{equation*}
\omega(e_{a} , f_{b}) = \sum_{c} \omega(\cN_{ac}\, f_{c} , f_{b}) = \sum_{c} \Im(\cN)_{ac}\omega(J f_{c} , f_{a}) = - \delta_{ab}\, ,
\end{equation*}

\noindent
which implies that $\Im(\cN)$ is a symmetric and positive definite $n\times n$ real matrix. Moreover, using the previous equation and the compatibility of $\cJ$ with $\omega$, we compute:
\begin{equation*}
0 = \omega(e_a , e_b) = \mathrm{Re}(\cN)_{ab} +\mathrm{Im}(\cN)_{ac} \omega(J(e_b),f_c) = \mathrm{Re}(\cN)_{ab} - \mathrm{Re}(\cN)_{ba}\, , 
\end{equation*}

\noindent
whence $\cN\colon V\to \Mat(n_v,\mathbb{C})$ is valued in the symmetric complex square matrices of positive-definite imaginary part. Hence, setting:
\begin{equation*}
\cR\eqdef - \mathrm{Re}(\cN)\, , \qquad \cI\eqdef \mathrm{Im}(\cN)\, ,
\end{equation*}

\noindent
we obtain a well-defined local electromagnetic structure $(\cR,\cI)$. Using again equation \eqref{eq:localperiod} to compute the action of $\cJ$ on the basis $\cE$ we obtain:
\begin{equation*}
\cJ(e_a) = -\cR_{ab} \cI^{-1}_{bc}\, e_c - \cR_{ab}\cI^{-1}_{bc}\cR_{cd}\, f_d  - \cI_{ad}\, f_d\, , \qquad \cJ(f_a) = \cI^{-1}_{ab}\, e_b + \cI^{-1}_{ab}\cR_{bc}\,f_c\, ,
\end{equation*}

\noindent
which is equivalent to \eqref{eq:Jlocaltaming}.  
\end{proof}

\begin{remark} The map $\cN\colon V\to \Mat(n_v,\mathbb{C})$
constructed in the proof of the previous proposition is called the
\emph{period matrix} in the literature and can be used to obtain a
convenient local formulation of bosonic supergravity, as we will
explain in Section \ref{sec:periodmatrix}.
\end{remark}

\noindent
For future convenience we introduce the following definition.

\begin{definition} A {\bf taming map} $\cJ$ on $V$ is a smooth map
$\cJ\colon V \to \Aut(\mathbb{R}^{2n_v})$ such that $\cJ\vert_p$ is a complex structure on $\mathbb{R}^{2n_v}$ which compatibly tames $\omega$,
where the latter is the standard symplectic structure on
$\mathbb{R}^{2n_v}$ as defined in \eqref{eq:omegastandard} in terms of
the canonical basis of $\mathbb{R}^{2n_v}$. We will denote the space
of all taming maps by $\mathfrak{J}_V(\mathbb{R}^{2n_v},\omega)$.
\end{definition}
 
\begin{prop}
\label{prop:1to1electromagnetic} There is a one-to-one correspondence
between taming maps $\cJ\colon V\to \Aut(\mathbb{R}^{2n_v})$ and local
electromagnetic structures $(\cR,\cI)$, i.e. there exists a canonical
bijection:
\begin{equation*}
\gamma\colon \mathfrak{E}_V \to \mathfrak{J}_V(\mathbb{R}^{2n_v},\omega)\, ,\qquad (\cR,\cI) \mapsto  \cJ= 
\begin{pmatrix} 
- \cI^{-1} \cR & \cI^{-1} \\
- \cI - \cR\cI^{-1}\cR & \cR \cI^{-1} 
\end{pmatrix} \, .  
\end{equation*}
\end{prop}

\begin{proof}
The statement follows directly from the proof of proposition \ref{prop:cNTaming}. The inverse:
\begin{equation*}
\gamma^{-1} \colon \mathfrak{J}_V(\mathbb{R}^{2n_v},\omega) \to \mathfrak{E}_V\, ,
\end{equation*}

\noindent
maps a taming map $\cJ\colon \in \mathfrak{J}_V(\mathbb{R}^{2n_v},\omega)$ to the electromagnetic structure $\gamma^{-1}(\cJ) = (\cR,\cI)$ given by:
\begin{equation*}
\cR \eqdef -\mathrm{Re}(\cN)\, , \qquad \cI \eqdef  \mathrm{Im}(\cN)\, ,
\end{equation*}

\noindent where $\cN$ is the complex matrix uniquely defined by $e_a =
\sum_b \cN_{ab} f_b$ in terms of the canonical symplectic basis $\cE =
(e_1 , \hdots , e_{n_v} , f_1,\hdots , f_{n_v})$ of
$(\mathbb{R}^{2n_v},\omega)$.
\end{proof}
\noindent By the previous proposition, an electromagnetic structure
can be equivalently described in terms of a taming map $\cJ$ and a
local scalar-electromagnetic structure can be denoted simply by pair
$(\cG,\cJ)$. This description of local electromagnetic structures is
particularly convenient for the global geometric formulation of
bosonic supergravity, as we discuss in section
\ref{sec:geometricsugra}.

The gauge fields $\left\{ A_{\Lambda}\right\}$ integrating $\left\{
G_{\Lambda}\right\}$ are usually referred to as the \emph{electric
	gauge fields}, whereas the one-forms $\left\{ A^{\Lambda}\right\}$
integrating $\left\{ F^{\Lambda}\right\}$ are usually referred as
their dual \emph{magnetic gauge fields}.

\begin{definition}
A vector-valued one-form $A\in \Omega^1(U,\mathbb{R}^{2n_v})$ is said to be {\bf twisted selfdual} with respect to the Lorentzian metric $g$, the scalar map $\phi\colon U\to \mathbb{R}^{n_s}$ and a taming map $\cJ$ if:
\begin{equation*}
\ast_g\cV = - \cJ(\phi)(\cV)\, , 
\end{equation*}
	
\noindent
where $\cV = \dd A$ and $\ast_g$ is the Hodge dual associated to $g$ and the given orientation on $U$. 
\end{definition}

\noindent
The global geometric interpretation of the local gauge fields
$(A^{\Lambda},A_{\Lambda})$ and their role in the formulation of
supergravity is investigated in \cite{LazaroiuShahbazi} through the
implementation of Dirac quantization. Since we will not consider Dirac
quantization in this review, we will consider instead the
\emph{classical} configuration space of local bosonic supergravity.

\begin{definition} The {\bf classical configuration space} $\Conf_U(\cG,\cJ)$ of
the local bosonic supergravity on $U$ associated to $(\cG,\cJ)$, where
$\cG$ is a Riemannian metric on $V$ and $\cJ$ is a taming map on $V$,
is defined as the following set:
\begin{equation*}
\Conf_U(\cG,\cJ) \eqdef \left\{ (g,\phi, \cV)\,\,\vert\,\, \ast\cV = - \cJ(\phi)(\cV)\, , \,\, g\in \mathrm{Lor}(U)\, , \,\, \phi \in C^{\infty}(U,V)\, ,\,\,  \cV \in \Omega^2(U,\mathbb{R}^{2n_v})\right\}\, .
\end{equation*}
 
\noindent
The {\bf solution space} $\Sol_U(\cG,\cJ)$ of the local bosonic supergravity associated to $(\cG,\cJ)$ is the subset of $\Conf_U(\cG,\cJ)$ whose elements satisfy the equations of motion of bosonic supergravity.
\end{definition}

\begin{definition} Let $(\cG,\cR,\cI)$ be a scalar-electromagnetic
structure. The {\bf scalar energy momentum tensor} associated to
$(\cG,\cR,\cI)$ is the following map:
\begin{eqnarray*}
\cT(\cG)\colon \Conf_U(\cG,\cR,\cI) \to \Gamma(T^{\ast}U\odot T^{\ast}U)\, ,\\
(g,\phi,\cV) \mapsto \cG_{ij}(\phi) \dd\phi^i \odot \dd \phi^j - \frac{1}{2} g \cG_{ij}(\phi) \partial_c \phi^i \partial^c \phi^j\, .
\end{eqnarray*}	

\noindent
The {\bf gauge energy momentum tensor} associated to $(\cG,\cR,\cI)$ is the following map:
\begin{eqnarray*}
\cT(\cR,\cI) = \cT(\cJ)\colon \Conf_U(\cG,\cR,\cI) \to \Gamma(T^{\ast}U\odot T^{\ast}U)\, , \\ 
(g,\phi,\cV) \mapsto  2 \cI_{\Lambda \Sigma}(\phi) F^{\Lambda}_{ac} F^{\Sigma c}_b \dd x^a \odot \dd x^b - \frac{1}{2} g \cI_{\Lambda \Sigma}(\phi)  F^{\Lambda}_{cd} F^{\Sigma cd}\, .
\end{eqnarray*}	

\noindent
The sum:
\begin{equation*}
\cT(\cG,\cR,\cI) \eqdef \cT(\cG) + \cT(\cR,\cI)\colon\Conf_U(\cG,\cR,\cI) \to \Gamma(T^{\ast}U\odot T^{\ast}U)\, ,
\end{equation*} 

\noindent
is the {\bf energy momentum tensor} of the local supergravity associated to $(\cG,\cR,\cI)$.
\end{definition}

\begin{remark}
The Einstein equations of the local bosonic supergravity associated to $(\cG,\cJ)$ can be written in terms of the energy momentum tensor simply as:
\begin{equation*}
\G^g = \cT(\cG,\cJ)(g,\phi,\cV)\, ,
\end{equation*}

\noindent
for $(g,\phi,\cV)\in \Conf_U(\cG,\cJ)$.
\end{remark}

\noindent
The gauge energy momentum tensor admits a convenient formulation in terms of the taming associated to $(\cR,\cI)$.

\begin{lemma}
\label{lemma:TGequivariant}
The following formula holds:
\begin{equation*}
\cT(\cJ)(g,\phi,\cV) =  \omega(\cV_{ac} , \cJ \cV_{b}^{\,\, c})\, \dd x^a\odot \dd x^b\, ,
\end{equation*}

\noindent
for every $(g,\phi,\cV)\in \Conf_U(\cG,\cR,\cI)$, where $\cJ = \gamma(\cR,\cI)$.
\end{lemma}

\begin{proof}
Write $\cV = (F,G(\phi))^t$. We compute:
\begin{equation*}
\omega(\cV_{ac} , \cJ \cV_{b}^{\,\, c}) = F^t_{ac}  (\cI + \cR \cI^{-1} \cR) F_{b}^{\,\, c}    + G^t_{ac}  \cI^{-1} G_{b}^{\,\, c} - 2 F^t_{ac}  \cR \cI^{-1} G_{b}^{\,\, c} = F^t_{ac} \cI F_{b}^{\,\, c}  + \ast F^t_{ac} \cI \ast F_{b}^{\,\, c}  \, .
\end{equation*}

\noindent
Using the relation:
\begin{equation*}
\ast F^t_{ac} \cI \ast F_{b}^{\,\, c} = F^t_{ac} \cI F_{b}^{\,\, c} - \frac{1}{2} F^t_{cd} \cI F^{cd} g_{ab}\, ,
\end{equation*}

\noindent
we conclude.
\end{proof}

\noindent Since the Maxwell equations reduce simply to the condition
$\dd\cV= 0$, every solution $\cV$ is locally integrable and thus we
can write:
\begin{equation}
\cV =
\begin{pmatrix} 
F^{\Lambda} \\
G_{\Lambda}
\end{pmatrix} 
= 
\begin{pmatrix} 
\dd A^{\Lambda}  \\
\dd A_{\Lambda}
\end{pmatrix}  \, , \qquad \Lambda = 1, \hdots , n_v\, .
\end{equation}


\subsection{Duality transformations of the local equations}
\label{sec:globalsymlocal}


A precise understanding of the group of duality transformations of the
local equations is crucial to construct the global geometric
formulation of bosonic extended supergravity. By 
\emph{duality transformations} we refer here to symmetries of the local
supergravity equations which do not involve diffeomorphisms of $U$,
that is, symmetries that \emph{cover} the identity on $U$. We are
especially interested in global symmetries of the equations of motion
that may not preserve the action functional
\eqref{eq:localsugra}. These symmetries extend to supergravity the
well-known electromagnetic duality transformations occurring in
standard electromagnetism \cite{Deser:1976iy,Deser:1981fr} and are a
key ingredient of bosonic supergravity in four Lorentzian dimensions
and its connection to string theory. Denote by
$\Sp(2n_v,\mathbb{R})\subset \Aut(\mathbb{R}^{2n_v})$ the group of
automorphisms preserving the standard symplectic form $\omega$ and by
$\Diff(V)$ the group of diffeomorphisms of $V$ preserving its fixed
orientation. In order to characterize the duality transformations of
local bosonic supergravity we recall first that the group
$\Diff(V)\times \Sp(2n_v , \mathbb{R})$ has a natural action
$\mathbb{A}$ on $\mathrm{Lor}(U)\times C^{\infty}(U,V)\times
\Omega^2(U,\mathbb{R}^{2n_v})$ given by:
\begin{eqnarray}
\label{eq:Aaction}
& \mathbb{A}\colon \Diff(V)\times \Sp(2n_v , \mathbb{R})\times  \mathrm{Lor}(U)\times C^{\infty}(U,V)\times \Omega^2(U,\mathbb{R}^{2n_v})\to \mathrm{Lor}(U)\times C^{\infty}(U,V)\times \Omega^2(U,\mathbb{R}^{2n_v})\, ,\nn \\
& (f,\mathfrak{A},g,\phi,\cV)\mapsto (g,f\circ \phi, \mathfrak{A}\,\cV)\, . 
\end{eqnarray}

\noindent
For every $(f,\mathfrak{A})\in \Diff(V)\times \Sp(2n_v , \mathbb{R})$ we define:
\begin{equation*}
\mathbb{A}_{f,\mathfrak{A}} \colon  \mathrm{Lor}(U)\times C^{\infty}(U,V)\times \Omega^2(U,\mathbb{R}^{2n_v})\to \mathrm{Lor}(U)\times C^{\infty}(U,V)\times \Omega^2(U,\mathbb{R}^{2n_v})\, , \,\, (g,\phi,\cV)\mapsto (g,f\circ \phi, \mathfrak{A}\,\cV)\, .
\end{equation*}

\noindent This action does not preserve the configuration space
$\Conf_U(\cG,\cJ)$ of a given local supergravity associated to the
local electromagnetic structure $(\cG,\cJ)$. There is however an
important subgroup of $\Diff(V)\times \Sp(2n_v , \mathbb{R})$, the
so-called \emph{U-duality group}, which does preserve both the
configuration and solution spaces of the given supergravity. In order to
characterize it we recall first the natural left-action of the group
$\Diff(V)\times \Sp(2n_v,\mathbb{R})$ on the set of taming maps, given
by:
\begin{equation*}
\cJ\mapsto \mathfrak{A}\, (\cJ\circ f^{-1})\, \mathfrak{A}^{-1}\, ,
\end{equation*}
	
\noindent
for every $(f,\mathfrak{A})\in \Diff(V)\times \Sp(2n_v,\mathbb{R})$ and every taming map $\cJ\colon V\to \mathrm{Aut}(\mathbb{R}^{2n_v})$. Given $(f,\mathfrak{A})\in \Diff(V)\times \Sp(2n_v,\mathbb{R})$ and a taming map $\cJ\colon V\to \mathrm{Aut}(\mathbb{R}^{2n_v})$, in the following we define:
\begin{equation*}
\cJ^f_\mathfrak{A} \eqdef \mathfrak{A}\, (\cJ\circ f^{-1})\, \mathfrak{A}^{-1}\, .
\end{equation*}

\begin{lemma}
\label{lemma:equivariancecT}
The total energy momentum tensor of local bosonic supergravity satisfies:
\begin{equation*}
\cT(f_{\ast}\cG,\cJ^f_{\mathfrak{A}})(\hat{g},\hat{\phi},\hat{\cV}) = \cT(\cG,\cJ)(g,\phi,\cV)\, ,
\end{equation*}

\noindent
where $(\hat{g},\hat{\phi},\hat{\cV}) = (g,f\circ \phi , \mathfrak{A}\cV)$ for $(f,\mathfrak{A})\in \Diff(V)\times \Sp(2n_v,\mathbb{R})$.  
\end{lemma}

\begin{proof}
We compute:
\begin{equation*}
\hat{\cG}_{ij}(\hat{\phi}) \partial_a \hat{\phi}^i \partial_b \hat{\phi}^j = \hat{\cG}_{ij}(f\circ\phi) \partial_a (f\circ\phi)^i \partial_b (f\circ\phi)^j = \hat{\cG}_{ij}(f\circ\phi) \partial_k f^i \partial_l f^j\, \partial_a \phi^k \partial_b \phi^l = \cG_{ij}(\phi)\partial_a \phi^k \partial_b \phi^l\, ,
\end{equation*} 

\noindent
where we have used that:
\begin{equation*}
\hat{\cG}_{ij} = (f_{\ast}\cG)_{ij} = \cG_{kl}\circ f^{-1}\, \partial_i (f^{-1})^k \partial_j (f^{-1})^l\, .
\end{equation*}

\noindent
This proves the statement for the scalar energy momentum tensor, that is:
\begin{equation*}
\cT(f_{\ast}\cG)(\hat{g},\hat{\phi},\hat{\cV}) = \cT(\cG)(g,\phi,\cV)\, .
\end{equation*}

\noindent
The statement for the gauge energy momentum tensor follows directly from Lemma \ref{lemma:TGequivariant}, which implies:
\begin{equation*}
\cT(\cJ^f_\mathfrak{A})(\hat{g},\hat{\phi},\hat{\cV}) = \cT(\cG)(g,\phi,\cV)\, ,
\end{equation*}

\noindent
and hence we conclude.
\end{proof}


 
\begin{thm}
\label{thm:equivsolutions}
For every $(f,\mathfrak{A})\in \Diff(V)\times \Sp(2n_v , \mathbb{R})$, the map $\mathbb{A}_{f,\mathfrak{A}}$ induces by restriction a bijection:
\begin{equation*}
\mathbb{A}_{f,\mathfrak{A}}\colon \Conf_U(\cG,\cJ)\to \Conf_U(f_{\ast}\cG,\cJ^f_A)\, , 
\end{equation*}

\noindent
such that it further restricts to a bijection of the corresponding spaces of solutions:
\begin{equation*}
\mathbb{A}_{f,\mathfrak{A}}\colon \Sol_U(\cG,\cJ)\to \Sol_U(f_{\ast}\cG,\cJ^f_{\mathfrak{A}})\, ,
\end{equation*}

\noindent
where $f_{\ast}\cG$ is the push-forward of $\cG$ by $f\colon V\to V$ and $\cJ^{f}_{\mathfrak{A}}  \eqdef \mathfrak{A}\,(\cJ\circ f^{-1})\mathfrak{A}^{-1}$
\end{thm}

\begin{remark}
If we consider a pair $(f , \mathfrak{U})\in \Diff(V)\times \Aut(\mathbb{R}^{2n_v})$, with $\mathfrak{A}$ not necessarily preserving $\omega$, then $\cJ^f_A$ is not guaranteed to be a taming map for the fixed standard symplectic structure $\omega$, a condition which is necessary for $\cJ^f_{\mathfrak{A}}$ to define a local electromagnetic structure. The group $\Diff(V)\times \Aut(\mathbb{R}^{2n_v})$ was discussed in \cite{Hull:1995gk} as the group of \emph{pseudo-dualities} of four-dimensional supergravity.
\end{remark}

\begin{proof}
We compute:
\begin{equation*}
\cJ(\phi)\,\cV = -\ast_g \cV \,\, \Leftrightarrow \,\, (\cJ\circ f^{-1})(f\circ \phi)\, \mathfrak{A}^{-1} \mathfrak{A} \cV = -\mathfrak{A}^{-1}\mathfrak{A} \ast_g  \cV \,\, \Leftrightarrow\,\, \mathfrak{A} \,(\cJ\circ f^{-1})(f\circ \phi)\,\mathfrak{A}^{-1} \mathfrak{A}\cV = -\ast_g (\mathfrak{A}\cV) \, ,
\end{equation*}

\noindent
which is equivalent to $\cJ^f_{\mathfrak{A}}(f\circ \phi)\, \mathfrak{A}\cV = - \ast_g (\mathfrak{A}\cV)$, whence:
\begin{equation*}
\mathbb{A}_{f,\mathfrak{A}}(g,\phi,\cV)\in \Conf_U(f_{\ast}\cG , \cJ^f_A)\, , \qquad \forall \,\,(g,\phi,\cV)\in\Conf_U(\cG,\cJ)\, .
\end{equation*} 

\noindent
The fact that $\mathbb{A}_{f,\mathfrak{A}}$ is a bijection is now clear. In order to prove that $\mathbb{A}_{f,\mathfrak{A}}$ preserves the corresponding spaces of solutions consider $(\hat{g}, \hat{\phi},\hat{\cV})\in \Conf_U(f_{\ast}\cG,\cJ^f_{\mathfrak{A}})$ such that:
\begin{equation*}
(\hat{g}, \hat{\phi},\hat{\cV}) = (g, f\circ \phi, \mathfrak{A} \cV)\, ,
\end{equation*} 

\noindent
for $(f,\mathfrak{A})\in \Diff(V,\cG)\times \Sp(2n_v,\mathbb{R})$ and $(g,\phi,\cV)\in \Sol_U(\cG,\cJ)$. Write now:
\begin{equation*}
\mathfrak{A} =  
\begin{pmatrix} 
a & b \\
c & d 
\end{pmatrix} \, ,
\end{equation*}

\noindent
in terms of $n_v\times n_v$ blocks and:
\begin{equation*}
\hat{\cV}
=   
\begin{pmatrix} 
\hat{F}  \\
\hat{G}(\hat{\phi})
\end{pmatrix} \, , \qquad 
\cV
=   
\begin{pmatrix} 
F  \\
G(\phi)
\end{pmatrix} \, .
\end{equation*}

\noindent
In particular $\hat{F} = a\, F + b\, G(\phi)$. The fact that the solution spaces of the Einstein equations are preserved by $(f,\mathfrak{A})$ follows directly from Lemma \ref{lemma:equivariancecT}, since it implies:
\begin{equation*}
\G^g = \cT(f_{\ast}\cG,\cJ^f_{\mathfrak{A}})(\hat{g},\hat{\phi},\hat{\cV}) = \cT(\cG,\cJ)(g,\phi,\cV)\, ,
\end{equation*}

\noindent
We consider now the scalar equations \eqref{eq:Scalarlocal}, which we evaluate on $(\hat{g}, \hat{\phi},\hat{\cV}) $ and rewrite for convenience as follows:
\begin{equation*}
\nabla^{g,\hat{\cG}(\hat{\phi})}_a \partial^a \hat{\phi}^i + \frac{\hat{\cG}^{ik}(\hat{\phi})}{2} (\partial_k \hat{\cR}_{\Lambda \Sigma}(\hat{\phi}) \hat{F}^{\Lambda}_{ab} \ast \hat{F}^{\Sigma ab} +   \partial_k \hat{\cI}_{\Lambda \Sigma}(\hat{\phi}) \hat{F}^{\Lambda}_{ab} \hat{F}^{\Sigma ab}) = 0\, ,
\end{equation*}

\noindent
where $\nabla^{g,\hat{\cG}(\hat{\phi})}$ is the product connection of the Levi-Civita connection of $g$ and the Levi-Civita connection of $\hat{\cG}(\hat{\phi})$. Using the equivariance properties of the Levi-Civita connection under metric pull-back we obtain:
\begin{equation*}
\nabla^{g,\hat{\cG}(\hat{\phi})}_a \partial^a \hat{\phi}^i   =  \partial_k f^i\,\nabla^{g,\cG(\phi)}_a \partial^a \phi^k\, .
\end{equation*}
 
\noindent 
On the other hand, a tedious calculation shows that:
\begin{eqnarray*}
&\frac{\hat{\cG}^{ik}(\hat{\phi})}{2} (\partial_k \hat{\cR}_{\Lambda \Sigma}(\hat{\phi}) \hat{F}^{\Lambda}_{ab} \ast \hat{F}^{\Sigma ab} +   \partial_k \hat{\cI}_{\Lambda \Sigma}(\hat{\phi}) \hat{F}^{\Lambda}_{ab} \hat{F}^{\Sigma ab})\\ 
& 
= \frac{\hat{\cG}^{ik}(\hat{\phi})}{2}  \omega(\hat{\cV}, \partial_k\hat{\cJ}(\hat{\phi}) \hat{\cV}) =
\partial_l f^i \frac{\cG^{lk}(\phi)}{4} \omega(\cV,
\partial_k\cJ(\phi) \cV) \, ,
\end{eqnarray*}

\noindent whence $(f,\cG)$ maps solutions of the scalar equations to
solutions of the scalar equations. Finally, the solution spaces of the
Maxwell equations \eqref{eq:Maxwelllocal} are also clearly preserved
since:
\begin{equation*}
\dd\hat{\cV} = \dd(\mathfrak{A}\cV) = \mathfrak{A}\dd\cV \, ,
\end{equation*}

\noindent
Therefore, $(\hat{g},\hat{\phi},\hat{\cV})\in \Sol_U(\cG,\cJ)$ if and only if $(g,\phi,\cV)\in \Sol_U(\cG,\cJ)$.
\end{proof}
 
\noindent The previous theorem can be used to characterize which
elements $(f,\mathfrak{U})\in \Diff(V)\times \Sp(2n_v,\mathbb{R})$
define through the action $\mathbb{A}$ symmetries of a the local
supergravity associated to a given local scalar-electromagnetic
structure $(\cG,\cJ)$. Denote in the following by $\Iso(V,\cG)$ the isometry group
of $\cG$.

\begin{cor}
\label{cor:equivarianceccJlocal}
Let $(f,\mathfrak{A})\in \Diff(V)\times \Sp(2n_v , \mathbb{R})$ such that $f\in \Iso(V,\cG)$ and:
\begin{equation}
\label{eq:localequivariancecJ}
\cJ^f_{\mathfrak{A}} = \mathfrak{A}\, (\cJ\circ f^{-1})\, \mathfrak{A}^{-1} = \cJ\, .
\end{equation}
	
\noindent Then $\mathbb{A}_{f,\mathfrak{A}}\colon \Conf_U(\cG,\cJ)\to
\Conf_U(\cG,\cJ)$ is a bijection of the configuration space of the
supergravity defined by the scalar-electromagnetic structure
$(\cG,\cJ)$. In particular, $\mathbb{A}_{f,\mathfrak{A}}$ preserves
the solution space $\Sol_U(\cG,\cJ)$, that is, it maps solutions to
solutions.
\end{cor}

\noindent A pair $(f,\mathfrak{A})\in \Iso(V,\cG)\times \Sp(2n_v ,
\mathbb{R})$ satisfying equation \eqref{eq:localequivariancecJ} will
be called a \emph{duality transformation}. It is easy to see that two
duality transformations $(f_1, \mathfrak{U}_1)$ and $(f_2,
\mathfrak{U}_2)$ can be composed in the natural way:
\begin{equation*}
(f_1, \mathfrak{A}_1)\circ (f_2, \mathfrak{A}_2) = (f_1\circ f_2, \mathfrak{A}_1\circ \mathfrak{A}_2)\, , 
\end{equation*}

\noindent
whence the set of all duality transformation of the local bosonic supergravity associated to $(\cG,\cJ)$, which we denote by $\U(\cG,\cJ)$, becomes naturally a group. 

\begin{definition}
\label{def:Udualitygroup}
The {\bf electromagnetic U-duality group}, or {\bf U-duality group} for short, of the local bosonic supergravity associated to $(\cG,\cJ)$ is given by:
\begin{equation}
\label{eq:localUduality}
\U(\cG,\cJ) \eqdef \left\{ (f,\mathfrak{A})\in  \Iso(V,\cG)\times \Sp(2n_v,\mathbb{R})\,\, \vert \,\, \mathfrak{A}\, \cJ\, \mathfrak{A}^{-1} = \cJ \circ f \right\}\, .
\end{equation}
\end{definition}

\begin{remark}
By corollary \ref{cor:equivarianceccJlocal}, for every element
$(f,\mathfrak{A})\in \U(\cG,\cJ)$, the map
$\mathbb{A}_{f,\mathfrak{A}}\colon \Conf_U(\cG,\cJ)\to
\Conf_U(\cG,\cJ)$ restricts to a bijection
$\mathbb{A}_{f,\mathfrak{A}}\colon \Sol_U(\cG,\cJ)\to
\Sol_U(\cG,\cJ)$.
\end{remark}

\noindent
Denote by $\Stab_\Sp(\cJ) \subset \Sp(2n_v,\mathbb{R})$ the subgroup of $\Sp(2n_v,\mathbb{R})$ preserving the given taming map $\cJ$, that is:
\begin{equation*}
\Stab_\Sp(\cJ) \eqdef \left\{ \mathfrak{A} \in \Sp(2n_v , \mathbb{R}) \,\, \vert \,\, \mathfrak{A}\, \cJ\, \mathfrak{A}^{-1} = \cJ\right\}\, .
\end{equation*}

\noindent
Then, for every electromagnetic structure $(\cG,\cJ)$ we have the following short exact sequence:
\begin{equation}
\label{eq:shortUduality1}
1 \to \Stab_\Sp(\cJ) \to \U(\cG,\cJ) \to \Iso_{pr}(V,\cG)\to 1\, ,
\end{equation}

\noindent where $\Stab_\Sp(\cJ)$ embeds in $\U(\cG,\cJ)$ through the
map $\mathfrak{A}\mapsto (\mathrm{Id},\mathfrak{A})$ and
$\Iso_{pr}(V,\cG)\subset \Iso(V,\cG)$ is the subgroup of the isometry
group of $(V,\cG)$ that is obtained by projecting $\U(\cG,\cJ)$ onto
its first component in the presentation \eqref{eq:localUduality}. On
the other hand, by definition we have a canonical surjective map:
\begin{equation*}
\U(\cG,\cJ) \to \Sp_{pr}(2n_v,\mathbb{R})\subset \Sp(2n_v,\mathbb{R})\, , \qquad (f,\mathfrak{A})\mapsto \mathfrak{A}\, ,
\end{equation*}

\noindent
fitting in the following short exact sequence:
\begin{equation}
\label{eq:shortUduality2}
1 \to \Stab_\Iso(\cJ) \to \U(\cG,\cJ) \to \Sp_{pr}(2n_v,\mathbb{R}) \to 1\, ,
\end{equation}

\noindent
where:
\begin{equation*}
\Stab_\Iso(\cJ) \eqdef \left\{ f \in \Iso(V,\cG) \,\, \vert \,\, \cJ\circ f = \cJ \right\}\, ,
\end{equation*}

\noindent is the stabilizer of $\cJ$ in $\Iso(V,\cG)$ and
$\Sp_{pr}(2n_v,\mathbb{R})$ the subgroup of $\Sp(2n_v,\mathbb{R})$
that is obtained by projecting $\U(\cG,\cJ)$ onto its second component
in the presentation \eqref{eq:localUduality}. All together, we obtain the following
proposition.

\begin{prop}
\label{prop:localUduality} The electromagnetic U-duality group
$\U(\cG,\cJ)$ of the local supergravity theory associated to the
electromagnetic structure $(\cG,\cJ)$ canonically fits into the short
exact sequences \eqref{eq:shortUduality1} and
\eqref{eq:shortUduality2}.
\end{prop}

\noindent The short exact sequences \eqref{eq:shortUduality1} and
\eqref{eq:shortUduality2} are very useful to compute the
electromagnetic U-duality group of a given local bosonic supergravity. In
particular, we obtain the following corollary.

\begin{cor} If $\Stab_\Sp(\cJ) = \mathrm{Id}$ then $\U(\cG,\cJ) =
\Iso_{pr}(V,\cG)\subset \Iso(V,\cG)$ canonically becomes a subgroup of
the isometry group of the scalar manifold $(V,\cG)$. If
$\Stab_\Iso(\cJ) = \mathrm{Id}$ is trivial then
\eqref{eq:shortUduality2} yields a canonical embedding
$\U(\cG,\cJ)\hookrightarrow \Sp(2n_v ,\mathbb{R})$ in the symplectic
group $\Sp(2n_v ,\mathbb{R})$. If both $\Stab_\Sp(\cJ) = \mathrm{Id}$
and $\Stab_\Iso(\cJ) = \mathrm{Id}$ then the U-duality group
$\U(\cG,\cJ)$ is canonically isomorphic to a subgroup of $\Iso(V,\cG)$ embedded
in $\Sp(2n_v,\mathbb{R})$.
\end{cor}

\noindent  
The previous corollary puts on firm grounds the validity of
a folklore statement made in the literature which states that the
\emph{U-duality group consists of a copy of the isometry group of the
scalar manifold into the symplectic group}. Before presenting some
examples it is useful to formulate local bosonic supergravity in terms
of the \emph{period matrix map}.
 

\subsection{The period matrix map}
\label{sec:periodmatrix}


For computational as well as conceptual purposes it is convenient to develop a local
formulation of the theory in terms of complexified field strength and
couplings, a formulation which gives rise to the concept of
\emph{period matrix} (whose name will be justified in a moment). We
define the \emph{complexified field strengths}:
\begin{equation*}
F^{+} \eqdef F - i\ast F\, , \qquad F^{-} \eqdef F + i\ast F\, ,
\end{equation*}

\noindent
in terms of which the gauge sector of the theory (associated to a given electromagnetic structure $(\cR,\cI)$) is conveniently written, using matrix notation, as follows:
\begin{equation*}
\mathrm{S}^v_l[g_U,\phi, A] \eqdef \frac{i}{4}\int_U\left\{ F^{+ T} \cN F^{+} - F^{- T} \cN^{\ast} F^{-} \right\} \nu_{g_U} \, , 
\end{equation*}

\noindent
where we have defined the \emph{period matrix map}:
\begin{equation*}
\cN \eqdef - \cR + i\cI \colon V \to \mathrm{Sym}(n_v,\mathbb{C})\, .
\end{equation*}

\noindent Since $(\cR,\cI)$ is a local electromagnetic structure, the
map $\cN$ is in fact a function on $V$ valued in Siegel upper space
$\mathbb{SH}(n_v)\subset \Mat(n_v,\mathbb{C})$ of square $n_v\times
n_v$ complex matrices with positive definite imaginary part. For ease
of notation, we define:
\begin{equation*}
\cN(\phi) = - \cR(\phi) + i\cI(\phi) \eqdef \cN \circ \phi \colon U \to \mathbb{SH}(n_v)\, ,
\end{equation*}

\noindent to which we will refer as the \emph{scalar} period matrix
map.

\begin{remark} The term \emph{period matrix} is motivated by the role
played by $\cN$ when the bosonic supergravity theory under
consideration corresponds to the effective theory of a Type-IIB
compactification on a Calabi-Yau three-fold $X$. In such situation,
$\cN$ establishes linear relations between the periods on the moduli
space of complex structures of $X$ with respect to a symplectic basis
of the third homology group of $X$ \cite{Fre:1995dw}.
\end{remark}

\noindent The twisted self-duality condition can be described in a
very natural manner by introducing complexified field
strengths. Recall that the complexification of the vector space of
real vector-valued two-forms $\Omega^2(U,\mathbb{R}^{2n_v})$ is given
by $\Omega^2(U,\mathbb{C}^{2n_v})$, the vector space of two-forms
taking values in the complex vector space $\mathbb{C}^{2n_v}$. We
define:
\begin{equation*}
G^{\pm} \eqdef\frac{1}{2} ( G \mp i \ast G)\in \Omega^2(U,\mathbb{C}^{n_v})\, , \qquad \cV^{\pm} \eqdef \frac{1}{2} (\cV \mp i \ast \cV) \in \Omega^2(U,\mathbb{C}^{2n_v})\, .
\end{equation*}

\noindent
Note that we have:
\begin{equation*}
\ast F^{\pm} = \pm i F^{\pm}\, .
\end{equation*}

\begin{prop}
\label{prop:cVcN}
A vector-valued two-form $\cV \in \Omega^2(U,\mathbb{R}^{2n_v})$ is twisted self-dual with respect to a scalar map $\phi$ and a taming map:
\begin{equation*}
\cJ= 
\begin{pmatrix} 
- \cI^{-1} \cR & \cI^{-1} \\
-\cI - \cR\cI^{-1}\cR & \cR \cI^{-1} 
\end{pmatrix} \colon V\to \Aut(\mathbb{R}^{2n_v}) \, ,
\end{equation*}

\noindent
if and only if:
\begin{equation}
\cV^{+} =
\begin{pmatrix} 
F^{+} \\
\cN^{\ast}(\phi) F^{+}
\end{pmatrix} 
\end{equation}
	\noindent
for a complex self-dual two-form $F^{+} = \frac{1}{2}(F - i\ast F)$, where $\cN= \cR+ i\cI\colon V \to \mathbb{SH}(n_v)$. 
\end{prop}

\begin{proof}
By Lemma \ref{lemma:twistedselfdual}, $\cV$ is twisted self-dual with respect to $\cJ$ if and only if:
\begin{equation*}
\cV
=   
\begin{pmatrix} 
F  \\
\cR F - \cI  \ast F   
\end{pmatrix} \, ,
\end{equation*}

\noindent
for $F\in \Omega^2(U,\mathbb{R}^{n_v})$. This equation can be easily shown to be equivalent to:
\begin{equation}
\cV^{+} =
\begin{pmatrix} 
F^{+} \\
\cN^{\ast}(\phi) F^{+}
\end{pmatrix} 
\end{equation}

\noindent
by computing $\cV^{+} = \frac{1}{2} (\cV - i \ast \cV)$. 
\end{proof}

\noindent
With these provisos in mind, we obtain:  
\begin{equation*}
G^{+} = \cN^{\ast}(\phi) F^{+}\, , \qquad G^{-} = \cN(\phi) F^{-}\, , \qquad \cV^{+} = (F^{+} , \cN^{\ast}F^{+})^t\, ,
\end{equation*}

\noindent
where the superscript $\ast$ denotes complex conjugation. These conditions are equivalent with:
\begin{equation*}
\cV = \frac{1}{2} (\cV^{+} + \cV^{-})\, ,
\end{equation*}

\noindent
being twisted self-dual with respect to the corresponding $\cJ$. 

\begin{remark} Since a period matrix $\cN$ is equivalent to the data
$(\cR,\cI)$, which in turn is equivalent to its
associated taming map $\cJ$, we will sometimes denote the
electromagnetic structure $(\cR,\cI)$ simply by $\cN$.
\end{remark}

\noindent Due to the fact that the period matrix $\cN$ takes values in
the Siegel upper space, the real symplectic group
$\mathrm{Sp}(2n_v,\mathbb{R})$ acts on $\cN$ through the natural left
action of $\mathrm{Sp}(2n_v,\mathbb{R})$ on $\mathbb{SH}(n_v)$ via
\emph{fractional transformations}. Recall that the fractional
transformation of $\tau\in\mathbb{SH}(n_v)$ by a matrix $\mathfrak{A}\in
\Sp(2n_v,\mathbb{R})$ is, by definition, given by:
\begin{equation*}
\mathfrak{A}\cdot \tau = \frac{ c + d \tau}{a + b \tau} \eqdef (c + d\tau) (a+b\tau)^{-1}\, , \qquad \tau \in\mathbb{SH}(n_v)\, ,
\end{equation*}

\noindent
where we wrote $A$ in $n_v\times n_v$ blocks as follows:
\begin{equation*}
\mathfrak{A} = \begin{pmatrix} 
a & b \\
c & d 
\end{pmatrix} \, .
\end{equation*}

\noindent This is the natural generalization of the action of
$\mathrm{Sl}(2,\mathbb{R})$ on the upper-half plane. Hence, a
symplectic matrix $\mathfrak{A}\in \mathrm{Sp}(2n,\mathbb{R})$ acts
point-wise on the period matrix $\cN(\phi) \colon U \to
\mathbb{SH}(n_v)$:
\begin{equation*}
\mathfrak{A} \cdot \cN(\phi) = \frac{c +  d \cN(\phi)}{a + b \cN(\phi)}\, .
\end{equation*}

\begin{remark}
More generally, $\Diff(V)\times \Sp(2n_v,\mathbb{R})$ has a natural left action on the set of period matrix maps as follows:
\begin{equation*}
\cN\mapsto \mathfrak{A}\cdot \cN\circ f^{-1}\, ,
\end{equation*}

\noindent for every $(f,\mathfrak{A})\in \Diff(V)\times
\Sp(2n_v,\mathbb{R})$ and every period matrix map $\cN\colon V\to
\mathbb{SH}(n_v)$.
\end{remark}

\noindent Since a period matrix map $\cN\colon V\to \mathbb{SH})(n_v)$
is equivalent to a taming map $\cJ\colon V\to \Aut(\mathbb{R}^{2 n_v})$ we can
describe the electromagnetic U-duality group defined in
\ref{def:Udualitygroup} in terms of a period matrix map. For this, we
recall first that, as an immediate consequence of Proposition
\ref{prop:1to1electromagnetic}, there exists a bijection:
\begin{equation*}
\mu \colon  \mathfrak{J}_V(\mathbb{R}^{2n_v},\omega)\to C^{\infty}(V,\mathbb{SH}(n_v)) \, ,
\end{equation*}

\noindent which, to every compatible taming $\cJ \in
\mathfrak{J}_V(\mathbb{R}^{2n_v},\omega)$ assigns the following period
matrix:
\begin{equation*}
\mu(\cJ) = \cR + i\cI \colon V\to \mathbb{SH}(n_v)\, ,
\end{equation*}

\noindent where $\gamma^{-1}(\cR,\cI) = \cJ$ is the electromagnetic
structure associated to $\cJ$ by means of the bijection $\gamma$. The
inverse of $\mu$ maps every period matrix $\tau= \mathrm{Re}(\tau) +
i\mathrm{Im}(\tau) \colon V\to \mathbb{SH}(n_v)$ to the taming defined
explicitly by equation \eqref{eq:Jlocaltaming} by identifying $\cR=
\mathrm{Re}(\tau)$ and $\cI= \mathrm{Im}(\tau)$.  

\begin{lemma}
\label{lemma:equivarianceelectromagnetic} The map $\mu \colon
\mathfrak{J}_V(\mathbb{R}^{2n_v},\omega)\to
C^{\infty}(V,\mathbb{SH}(n_v))$ is equivariant with respect to the
natural action of $\Diff(V) \times \Sp(2n_v,\mathbb{R}$ on
$\mathfrak{J}_V(\mathbb{R}^{2n_v},\omega)$ and
$C^{\infty}(V,\mathbb{SH}(n_v)$, respectively. That is, the following
relation holds:
\begin{equation*}
\mu(\cJ^f_{\mathfrak{A}}) = \mathfrak{A}\cdot \mu(\cJ)\circ f^{-1}\, ,
\end{equation*}

\noindent
for every $(f,\mathfrak{A})\in \Diff(V)\times \Sp(2n_v,\mathbb{R})$.
\end{lemma}
 
\begin{prop} Let $\cJ$ be a taming map and set $\cN\eqdef
\mu(\cJ)$. The electromagnetic U-duality group $\U(\cG,\cJ)$ of the
local bosonic supergravity associated to $(\cG,\cJ)$ is canonically
isomorphic to:
\begin{equation}
\label{eq:UdualitycN}
\U(\cG,\cN) \eqdef \left\{ (f,\mathfrak{A})\in  \Iso(V,\cG)\times \Sp(2n_v,\mathbb{R})\,\, \vert \,\, \mathfrak{A}\cdot \cN = \cN\circ f \right\}\, , \quad \cN = \mu(\cJ)\, ,
\end{equation}

\noindent
through the identity map $\U(\cG,\cN) \ni (f,\mathfrak{A})\mapsto (f,\mathfrak{A})\in \U(\cG,\cJ)$.
\end{prop}

\begin{proof} It is enough to note that $(f,\mathfrak{A})\in
\Diff(V)\times \Sp(2n_v,\mathbb{R})$ satisfies $\cJ^f_{\mathfrak{A}} =
\cJ$ if and only if $\mu(\cJ^f_{\mathfrak{A}}) = \mu(\cJ)$, which in
turn is equivalent to:
\begin{equation*}
\mathfrak{A}\cdot\mu(\cJ) = \mu(\cJ)\circ f \, ,
\end{equation*}

\noindent
by Lemma \ref{lemma:equivarianceelectromagnetic}. 
\end{proof}

\begin{remark} Equation \eqref{eq:UdualitycN}, which defines the duality
group in terms of the period matrix $\cN$, can be alternatively
obtained as follows, which is the way in which this group is usually
described in the literature. Consider:
\begin{equation*}
(\hat{g}, \hat{\phi},\hat{\cV}) = (g, f\circ \phi, \mathfrak{A} \cV)\, ,
\end{equation*} 

\noindent
for $(f,\mathfrak{A})\in \Diff(V,\cG)\times \Sp(2n_v,\mathbb{R})$ and $(g,\phi,\cV)\in \Sol_U(\cG,\cJ)$. Write:
\begin{equation*}
\mathfrak{A} =  
\begin{pmatrix} 
a & b \\
c & d 
\end{pmatrix} \, .
\end{equation*}

\noindent
By Proposition \ref{prop:cVcN}, we have $\cV^{+} = (F^{+} ,
\cN^{\ast}(\phi) F^{+})$ but in general for an arbitrary
$(f,\mathfrak{A})$ there will exist no period matrix $\hat{\cN}\colon
V\to \mathbb{SH}(n_v)$ such that $\hat{\cV}^{+} = (\hat{F}^{+} ,
\hat{\cN}^{\ast}(\hat{\phi}) \hat{F}^{+})$. Imposing that such period
matrix $\hat{\cN}$ exists we obtain:
\begin{equation*}
\hat{\cN} = \mathfrak{A}\cdot\cN\circ f^{-1}\, ,
\end{equation*}

\noindent
and the condition appearing in \eqref{eq:UdualitycN} follows now by imposing $\hat{\cN} = \cN$. 
\end{remark}

\noindent Using the previous canonical identification between
$\U(\cG,\cN)$ and $\U(\cG,\cJ)$ we obtain short exact sequences
for $\U(\cG,\cN)$ analogous to \eqref{eq:shortUduality1} and
\eqref{eq:shortUduality2}, namely:
\begin{eqnarray}
\label{eq:shortUdualitycN}
& 1 \to \Stab_\Sp(\cN) \to \U(\cG,\cN) \to \Iso_{pr}(V,\cG)\to 1\, , \nonumber\\ 
& 1 \to \Stab_\Iso(\cN) \to \U(\cG,\cN) \to \Sp_{pr}(2n_v,\mathbb{R}) \to 1\, ,
\end{eqnarray}
\noindent
where:
\begin{equation*} 
\Stab_\Sp(\cN) \eqdef \left\{ \mathfrak{A} \in
\Sp(2n_v , \mathbb{R}) \,\, \vert \,\, \mathfrak{A}\cdot \cN  = \cN\right\}\, ,\quad \Stab_\Iso(\cN) \eqdef
\left\{ f \in \Iso(V,\cG) \,\, \vert \,\, \cN\circ f = \cN \right\}
\end{equation*}

\noindent We consider now several examples of importance in
supergravity, which are more conveniently studied using the period
matrix map rather than the taming map.

\begin{example} Assume that $\cN\in \mathbb{SH}(n_v)$ is a constant
matrix. Then:
\begin{equation*} \cN\circ f = \mathfrak{A}\cdot\cN\, ,
\end{equation*}

\noindent for every $f\in \Iso(V,\cG)$ and $\mathfrak{A} \in
\Stab_{\Sp}(\cN)\subset \Sp(2n_v,\mathbb{R})$. Since the action of
$\Sp(2n_v,\mathbb{R})$ on $\mathbb{SH}(n_v)$ is transitive with
stabilizer isomorphic to the unitary group $\U(n_v) \subset \Sp(2n_v ,
\mathbb{R})$ and $\cN$ is assumed to be constant, the conjugacy class
of the stabilizer of $\cN$ in $\Sp(2n_v,\mathbb{R})$ is independent of
$\cN$. Therefore:
\begin{equation*} \U(\cG,\cN) \simeq \Iso(V,\cG)\times \U(n_v)\, ,
\end{equation*}

\noindent
is by Proposition \ref{prop:localUduality} the corresponding U-duality group.
\end{example}

\begin{example}
Set $n_v = 1$. We have $\mathbb{SH}(1) = \mathbb{H}$. Furthermore, assume $V = \mathbb{H}$ is equipped with its Poincar\'e metric $\cG$. Take $\cN\colon \mathbb{H} \to \mathbb{H}$ to be the identity map, that is, $\cN(\tau) = \tau$ where $\tau$ is the global coordinate on $\mathbb{H}$. Notice that this particular period matrix occurs in pure $N=4$ four-dimensional supergravity \cite{Bellorin:2005zc}. We have $\Iso(\mathbb{H},\cG) = \mathrm{PSl}(2,\mathbb{R})$ acting on $\mathbb{H}$ through fractional transformations. Hence:
\begin{equation*}
\cN\circ f(\tau) = f\cdot\tau = f\cdot \cN(\tau) = \hat{f}\cdot \cN(\tau)\, , \qquad \forall \,\, f\in\Iso(\mathbb{H},\cG)\, ,
\end{equation*}
	
\noindent
where $\hat{f}\in \mathrm{Sl}(2,\mathbb{R})$ denotes any lift of $f\in \mathrm{PSl}(2,\mathbb{R})$ to $\mathrm{Sl}(2,\mathbb{R})$. This implies that $\Iso_{pr}(\cM,\cG) = \Iso(\cM,\cG)$. On the other hand, a direct computation shows that:
\begin{equation*}
\Stab_{\Sp}(\cN) = \mathbb{Z}_2 = \left\{ \mathrm{Id}, - \mathrm{Id}\right\} \subset \mathrm{Sl}(2,\mathbb{R})\, .
\end{equation*}
	
\noindent
Therefore, by Proposition \ref{prop:localUduality} we have the following short exact sequence:
\begin{equation}
\label{eq:Z2extension}
1 \to \mathbb{Z}_2 \to \U(\cG,\cN) \to \mathrm{PSl}(2,\mathbb{R}) \to 1\, ,
\end{equation}
	
\noindent
which yields a central extension of $\U(\cG,\cN)$. Using now the fact that $\Stab_\Iso(\cN) = \mathrm{Id}$ we conclude that the electromagnetic U-duality group is $\U(\cG,\cN) = \mathrm{Sl}(2,\mathbb{R})$ and \eqref{eq:Z2extension} is indeed a non-trivial central extension of the isometry group of the scalar manifold. 
\end{example}

\begin{example}
Take $n_v =2$ and set $V = \mathbb{H}$ equipped with its Poincar\'e metric $\cG$. Consider the period matrix $\cN\colon \mathbb{H} \to \mathbb{SH}(2)$ defined as follows:
\begin{equation*}
\cN(\tau) = \begin{pmatrix} 
\tau & 0 \\
0 & - \frac{1}{\tau} 
\end{pmatrix} \, ,
\end{equation*}
	
\noindent
where $\tau$ is the global coordinate on $\mathbb{H}$. Clearly $\cN$ is symmetric. Furthermore, its imaginary part is positive definite:	
\begin{equation*}
\Im \cN(\tau) = \begin{pmatrix} 
\Im(\tau) & 0 \\
0 & \frac{\Im(\tau)}{\vert\tau\vert^2} 
\end{pmatrix} \, ,
\end{equation*}
	
\noindent whence $\cN$ is a well-defined period matrix. In fact,
$\cN$ is the period matrix occurring in the \emph{axio-dilaton model}
of $\cN=2$ supergravity, see for example \cite[Section
2]{Galli:2011fq} for more details. It is unambiguously fixed by
supersymmetry and in particular by the projective special K\"ahler
structure of the scalar manifold of the theory. As in the previous
example, we have:
\begin{equation*}
\Iso(V,\cG) = \mathrm{PSl}(2,\mathbb{R})\, , 
\end{equation*}
	
\noindent
acting through fractional transformations. Let $\mathfrak{A}\in\mathrm{Sp}(4n_v,\mathbb{R})$. A quick computation shows that:
\begin{equation*}
\mathfrak{U}\cdot \cN  = \cN\, , 
\end{equation*}	
	
\noindent
if and only if:
\begin{equation*}
\mathfrak{U} = (u,u) \in \SO(2) \times \SO(2) \hookrightarrow \Sp(4,\mathbb{R})\, ,
\end{equation*}
	
\noindent where the embedding is diagonal. Hence: $\Stab_\Sp(\cN) =
\U(1)$ diagonally embedded in $\Sp(4,\mathbb{R})$. On other hand, it
can be seen that $\Iso_{pr}(\mathbb{H},\cG) =
\mathrm{PSl}(2,\mathbb{R})$ and hence, the electromagnetic U-duality
group fits into the following short exact sequence:
\begin{equation*}
1 \to \U(1) \to \U(\cG,\cN) \to \mathrm{PSl}(2,\mathbb{R}) \to 1\, ,
\end{equation*} 
	
\noindent by Proposition \ref{prop:localUduality}. Moreover, it can be
easily verified that $\Stab_\Iso(\cN) = \mathrm{Id}$. Hence, the
U-duality group is canonically embedded as $\U(\cG,\cN) =
\Sp_{pr}(4,\mathbb{R})\hookrightarrow \Sp(4 , \mathbb{R})$, which
provides an explicit realization of the electromagnetic U-duality
group in $\Sp(4,\mathbb{R})$.
\end{example}

\begin{example}
Take $n_v =2$ and set $V = \mathbb{H}$ equipped with its Poincar\'e metric $\cG$. Consider the period matrix $\cN\colon \mathbb{H} \to \mathbb{SH}(2)$ defined as follows:
\begin{equation*}
\cN(\tau) = \begin{pmatrix} 
\frac{\tau^2}{2} (\tau + 3 \tau^{\ast})& -\frac{3}{2} \tau (\tau + \tau^{\ast}) \\
-\frac{3}{2} \tau (\tau + \tau^{\ast}) &  3 (\tau + \tau^{\ast}) + \frac{3}{2} (\tau - \tau^{\ast}) 
\end{pmatrix} \, ,
\end{equation*}
	
\noindent
where $\tau$ is the global coordinate on $\mathbb{H}$. Clearly $\cN$ is symmetric. Its imaginary part can be computed to be:	
\begin{equation*}
 \Im \cN(\tau) = \begin{pmatrix} 
\Im(\tau)^3 + 3 \Re(\tau)^2 \Im(\tau)  & -3 \Re(\tau) \Im(\tau) \\
- 3 \Re(\tau) \Im(\tau) & 3 \Im(\tau)
\end{pmatrix} \, ,
\end{equation*}
	
\noindent
It is easy to see that $ \Tr(\Im \cN(\tau)) > 0$ and $ \det(\Im \cN(\tau)) > 0$ whence $\Im(\cN)$ is positive definite and $\cN$ is well-defined as a period matrix. In fact, $\cN$ is the period matrix occurring in the \emph{$t^3$ model} of $\cN=2$ supergravity, see \cite[Section 7]{Ortin} for more details, which is a particularly important supergravity model in the context of Type-II compactifications on Calabi-Yau three-folds. It is unambiguously fixed by supersymmetry and in particular by the projective special K\"ahler structure of the scalar manifold of the theory. As in the previous example, we have:
\begin{equation*}
\Iso(V,\cG) = \mathrm{PSl}(2,\mathbb{R})\, , 
\end{equation*}
	
\noindent
acting through fractional transformations on $\tau$. We have $\cN\circ f = \cN$ for $f\in \Iso(V,\cG)$ if and only if $f$ is the identity, whence $\Stab_\Iso(\cN) = \mathrm{Id}$ and $\U(\cG,\cN)$ embeds in $\Sp(4,\mathbb{R})$. Let $\mathfrak{A} \in\mathrm{Sp}(4n_v,\mathbb{R})$. A tedious computation shows now that:
\begin{equation*}
\mathfrak{U}\cdot \cN  = \cN\, , 
\end{equation*}	
	
\noindent
if and only if $\mathfrak{A} = \mathrm{Id}$ as well as $\Iso_{pr}(V,\cG) = \Iso(V,\cG)$. Therefore, the U-duality group is isomorphic to $\mathrm{PSl}(2,\mathbb{R})$ and is canonically embedded in $\Sp(4,\mathbb{R})$. 
\end{example}

\begin{remark}
Definition \ref{def:Udualitygroup} and Proposition \ref{prop:localUduality} should be compared with the characterization of U-duality groups already existing in the supergravity literature, see for instance \cite{Aschieri:2008ns,Fre:1995dw}. The general approach considers a \emph{fixed embedding} of the isometry group in the symplectic group of the appropriate dimension in such a way that for each isometry of $(V,\cG)$ there exists a unique symplectic transformation satisfying Equation \eqref{eq:localequivariancecJ} and no isometry leaves $\cJ$ (or the period matrix) invariant. This immediately implies by assumption that $\Stab_{\Sp}(\cN) = \Stab_\Iso = \mathrm{Id}$ and hence such U-duality group is simply a copy of the isometry group of $(V,\cG)$ inside $\Sp(2n_v,\mathbb{R})$. This is in general not the case for the U-duality group introduced in Definition \ref{def:Udualitygroup}, which therefore differs from the one considered in the literature. This difference becomes more dramatic when considering the global formulation of the theory, see Section \ref{sec:globalautgroup}. 
\end{remark}


\section{Geometric bosonic supergravity}
\label{sec:geometricsugra}


In this section we describe the global geometric formulation of the
generic bosonic sector of supergravity on an oriented four-manifold
$M$, to which we will refer simply as \emph{geometric bosonic
supergravity}, or geometric supergravity for short, following the
terminology introduced in \cite{Cortes:2018lan}. The key points we
have considered when constructing geometric bosonic supergravity are
the following:

\

\begin{itemize}
	\item We have required geometric bosonic supergravity to be
defined in terms of global differential operators acting on the spaces
of sections of the appropriate fiber bundles. This is specially
important to study the global structure of supergravity solutions and the
associated moduli spaces.
	
	\
	
	\item We have required geometric supergravity to implement the
electromagnetic U-duality groups described in Section
\ref{sec:globalsymlocal}, in the sense that it must be possible to
understand the theory as being the result of \emph{gluing} the local
theories introduced in Section \ref{sec:localsugra} by means of a
\u{C}ech one cocycle valued in the symplectic group $\Sp(2n,\mathbb{R})$. 
This point is specially important for geometric supergravity to describe 
\emph{supergravity U-folds} in a differential-geometric context, as explained in
\cite{Lazaroiu:2016iav}, which is particularly relevant for string theory
applications.
\end{itemize}

\

\noindent Since we are mainly interested in the mathematical structure
of the gauge sector of the theory (which is responsible for the
existence of a \emph{symplectic duality structure}), we
will assume for simplicity that the theory is coupled to a standard
non-linear sigma model, instead of the more general notion of
\emph{section} sigma model considered in \cite{Lazaroiu:2017qyr}.

Instead of going through the process of constructing geometric bosonic
supergravity we present it in its final form and we discuss
its most interesting features. A geometric bosonic supergravity with metrically
trivial section sigma model is determined by the following data \cite{Lazaroiu:2016iav}:

\

\begin{itemize}
	\item An oriented and complete Riemannian manifold
$(\cM,\cG)$, the so-called \emph{scalar manifold} of the theory.
	
	\
	
	\item A triple $\Delta \eqdef (\cS,\omega,\cD)$ consisting of
a vector bundle $\cS$ over $\cM$ endowed with the symplectic pairing $\omega$ and
the flat symplectic connection $\cD$. We denote the complexification
of $\Delta = (\cS,\omega,\cD)$ by $\Delta_{\mathbb{C}} =
(\cS_{\C},\omega_{\C},\cD_{\C})$.
	
	\
	
	\item A compatible taming $\cJ$ on $(\cS,\omega,\cD)$, that
is, a complex structure on $\cS$ satisfying:
	\begin{equation*} \omega(\cJ s_1,\cJ s_2) = \omega(s_1 ,s_2)\,
, \qquad Q(s,s) \eqdef \omega(s , \cJ s) \geq 0\, ,
	\end{equation*}
	
	\noindent for all $s_1, s_2, s \in \cS$, with $Q(s,s) = 0$ if
and only if $s=0$. We will denote by $\Theta \eqdef (\Delta,\cJ)$
a pair consisting of a flat symplectic vector bundle $\Delta$ equipped
with a compatible taming $\cJ$.
\end{itemize}

\

\noindent Following the terminology of \cite{Lazaroiu:2016iav}, and given
a scalar manifold $(\cM,\cG)$, we will refer to $\Delta$ as a
\emph{duality structure} and to $\Theta = (\Delta,\cJ)$
as an \emph{electromagnetic structure}. The notion of morphism of
duality structures and electromagnetic structures is the natural one
given by a morphism of vector bundle preserving the relevant data
data, see \cite{Lazaroiu:2016iav} for more details. Finally, we will
refer to a scalar manifold $(\cM,\cG)$ together with a choice of
electromagnetic structure $\Theta$ as a scalar-electromagnetic
structure $\Phi$:
\begin{equation*} \Phi \eqdef (\cM,\cG,\Theta)\, .
\end{equation*}

\noindent A choice of duality structure $\Delta$ together with a
scalar manifold $(\cM,\cG)$ will be referred to as a
\emph{scalar-duality} structure, being denoted simply by
$(\cM,\cG,\Delta)$. Isomorphism classes of duality structures on a
fixed scalar manifold $\cM$ are in general not unique and depend on
the fundamental group of $\cM$. By the standard theory of flat vector
bundles, isomorphism classes of duality structures are in one to one
correspondence with the character variety:
\begin{equation*} \mathfrak{M}_d \eqdef \Hom(\pi_1(\cM),
\mathrm{Sp}(2n_v,\mathbb{R}))/\mathrm{Sp}(2n_v,\mathbb{R}) \, .
\end{equation*}

\begin{remark} The fact that character varieties yield in general
\emph{continuous} moduli spaces implies that one can construct an
uncountable infinity of inequivalent geometric bosonic supergravites,
all of which are however locally equivalent.
\end{remark}

\noindent A duality structure $\Delta$ can be \emph{trivial} in two
generally inequivalent senses. We say that $\Delta$ is
\emph{symplectically trivial} if $(\cS,\omega)\in \Delta$ is
symplectically trivial, that is, if it admits a global symplectic
frame. On the other hand, we will say that $\Delta$ is \emph{holonomy
trivial} if $\Delta$ is symplectically trivial and the holonomy of
$\cD$ is in addition trivial. Note that if $\cM$ is simply connected
every duality structure is symplectically trivial and holonomy
trivial.


\subsection{Geometric background}


Let $\Phi = (\cM,\cG,\Theta)$ be a scalar-electromagnetic
structure. Smooth maps from $M$ to $\cM$ will be called \emph{scalar
maps}. For every scalar map $\varphi\colon M\to \cM$ we use the
superscript $~^\varphi$ to denote bundle pull-back by $\varphi$. For
instance, $\Delta^{\varphi}$ will denote the pull-back of $\Delta$ by
$\varphi$, which defines a flat symplectic vector bundle over $M$, and
$\Theta^{\varphi}$ will denote the pull-back of $\Theta$ by $\varphi$,
respectively. For every Lorentzian metric $g$ on $M$ and scalar map
$\varphi$ we define an isomorphism of vector bundles:
\begin{equation*} \star_{g, \cJ^{\varphi}} \colon \Lambda T^{\ast}M
\otimes \cS^{\varphi} \to \Lambda T^{\ast}M \otimes \cS^{\varphi}\, ,
\end{equation*}

\noindent through the following equation:
\begin{equation*} \star_{g,\cJ^{\varphi}}(\alpha\otimes s) = \ast_g
\alpha\otimes \cJ^{\varphi}(s)\, , \qquad \alpha\in \Lambda
T^{\ast}M\, , \quad s\in \cS^{\varphi}
\end{equation*}

\noindent on homogeneous elements. Since the square of the Hodge
operator on two-forms is minus the identity, we obtain by restriction
an involutive isomorphism of vector bundles:
\begin{equation*} \star_{g,\cJ^{\varphi}} \colon \Lambda^2 T^{\ast}M
\otimes \cS^{\varphi} \to \Lambda^2 T^{\ast}M \otimes \cS^{\varphi}\,
,
\end{equation*}

\noindent that is, $\star_{g,\cJ^{\varphi}}^2 = 1$. Hence we can
split the bundle of two-forms taking values in $\cS^{\varphi}$ in
eigenbundles of $\star_{g,\cJ^{\varphi}}$:
\begin{equation*} \Lambda^2 T^{\ast}M \otimes \cS^{\varphi} =
(\Lambda^2 T^{\ast}M \otimes \cS^{\varphi})_{+} \oplus (\Lambda^2
T^{\ast}M \otimes \cS^{\varphi})_{-}\, ,
\end{equation*}

\noindent where the subscript denotes the corresponding
eigenvalue. The associated spaces of sections will be denoted
accordingly by:
\begin{equation*} \Omega^2(M,\cS^{\varphi}) =
\Omega^2_{+}(M,\cS^{\varphi}) \oplus \Omega^2_{-}(M,\cS^{\varphi})\, .
\end{equation*}

\begin{definition} Elements of $\Omega^2_{+}(M,\cS^{\varphi})$ will be
called {\bf twisted selfdual two-forms} and elements of
$\Omega^2_{-}(M,\cS^{\varphi})$ will be called {\bf twisted
anti-selfdual two-forms}.
\end{definition}

\noindent The flat symplectic connection $\cD^{\varphi}\in
\Delta^{\varphi}$ defines a canonical exterior covariant derivative
for forms on $M$ taking values in $\cS^{\varphi}$, which we denote by:
\begin{equation*} \dd_{\cD^{\varphi}}\colon \Omega^k(M,\cS^{\varphi})
\to \Omega^{k+1}(M,\cS^{\varphi})\, ,
\end{equation*}

\noindent where $k=0,\hdots 4$. Since $\cD^{\varphi}$ is flat, the
operator $\dd_{\cD^{\varphi}}$ is a coboundary operator on the complex
of forms taking values in $\cS^{\varphi}$. We denote the associated
cohomology groups by $H^{k}(M,\Delta^{\varphi})$ and the
corresponding total cohomology by $H(M,\Delta^{\varphi})$. Denote by:
\begin{equation*} \mathfrak{G}_{\Delta}(U) \eqdef \left\{ s \in
\Gamma(U,\cS^{\varphi}) \,\, \vert\,\, \cD^{\varphi} s = 0\right\}\, ,
\qquad U\subset M\, ,
\end{equation*}

\noindent the sheaf of smooth flat sections of $\Delta$. This is a
locally constant sheaf of symplectic vector spaces of rank $2n_{v}$,
whose stalk is isomorphic to the typical fiber of $\Delta$. There
exists a natural isomorphism of graded vector spaces:
\begin{equation*} H(M,\Delta^{\varphi}) \simeq
H(M,\mathfrak{G}_{\Delta})\, ,
\end{equation*}

\noindent where $H(M,\mathfrak{G}_{\Delta})$ denotes the sheaf
cohomology of $\mathfrak{G}_{\Delta}$.

Note that the definition of electromagnetic structure $\Theta =
(\Delta,\cJ)$ does not require $\cD\in \Delta$ to be compatible with
$\cJ$; the case when they are non-compatible is in fact crucial for the
correct description of geometric bosonic supergravity. The failure of
$\cD$ to be compatible with $\cJ$ is measured by the \emph{fundamental
form} of an electromagnetic structure.

\begin{definition} Let $\Phi =(\cM,\cG,\Theta)$ be a
  scalar-electromagnetic structure. The {
    \em fundamental form} $\Psi$
of $\Theta$ is the following one-form on $\cM$ taking
values in $\End(\cS)$:
\begin{equation*} \Psi \eqdef \cD\cJ \in
\Omega^1(\cM,\End(\cS))\, .
\end{equation*}
\end{definition}
 
\begin{remark}
It is not hard to see that $\Psi(X)\in \Gamma(\End(\cS))$, $X\in T\cM$, is an anti-linear self-adjoint endomorphism of the Hermitian vector bundle $(\cS,Q,\cJ)$. 
\end{remark}

\begin{definition}
An electromagnetic structure $\Theta$ is called {\em unitary} if $\Psi = 0$.
\end{definition}

\noindent To define geometric bosonic supergravity we need to
introduce three natural operations on tensors taking values in a
vector bundle.  These operations depend on the choice of
electromagnetic structure $\Theta$.

\begin{definition}
\label{def:tep} The {\em twisted exterior pairing}
$(\cdot,\cdot)_{g,Q^{\varphi}}$ is the unique pseudo-Euclidean scalar
product on $\Lambda T^{\ast}M\otimes\cS^{\varphi}$ satisfying:
\begin{equation*}
(\rho_1\otimes s_1,\rho_2\otimes s_2)_{g,Q^{\varphi}}=(\rho_1,\rho_2)_g Q^{\varphi}(s_1,s_2)\, ,
\end{equation*}
for any $\rho_1,\rho_2\in \Omega(M)$ and any $s_1,s_2\in \Gamma(\cS^{\varphi})$, where $(-, -)_g$ denotes the scalar product induced by $g$ on tensors over $M$. Recall that $Q(s_1,s_2) = \omega(s_1,\cJ s_2)$ and the superscript denotes pull-back by $\varphi$.
\end{definition}

\noindent For any vector bundle $W$ over $M$, we trivially extend the
twisted exterior pairing to a $W$-valued pairing (which for simplicity
we denote by the same symbol) between the bundles $W\otimes \Lambda
T^{\ast}M\otimes \cS^{\varphi}$ and $\Lambda T^{\ast}M\otimes
\cS^{\varphi}$:
\begin{equation*}
(w \otimes \eta_1,\eta_2)_{g,Q^{\varphi}} \eqdef  w \otimes (\eta_1,\eta_2)_{g,Q^{s}}\, , \quad\forall\, w\in \Gamma(W)\, , \quad \forall\, \eta_1,\eta_2\in \Lambda T^{\ast}M \otimes\cS^{\varphi}\, .
\end{equation*}

\begin{definition}
The {\em inner $g$-contraction of (2,0) tensors} is the bundle morphism $\oslash_g:(\otimes^2T^\ast M)^{\otimes 2}\rightarrow\otimes^2 T^\ast M$ uniquely determined by the condition:
\begin{equation*}
(\alpha_1\otimes\alpha_2)\oslash_g (\alpha_3\otimes \alpha_4)=(\alpha_2,\alpha_4)_g\alpha_1\otimes \alpha_3\, , \quad \forall\, \alpha_1, \alpha_2, \alpha_3, \alpha_4\in T^{\ast}M\, .
\end{equation*}

\noindent
We define the \emph{inner $g$-contraction of two-forms} to be the restriction of $\oslash_g$ to $\wedge^2 T^\ast M \otimes \wedge^2 T^\ast M\subset (\otimes^2T^\ast M)^{\otimes 2}$.
\end{definition}

\begin{definition}
We define the {\em twisted inner contraction} of $\cS^{\varphi}$-valued two-forms to be the unique morphism of vector bundles:
\begin{equation*}
\oslash_{Q}\colon\Lambda^2 T^{\ast}M\otimes\cS^{\varphi}\times_M \Lambda^2 T^{\ast}M\otimes\cS^{\varphi}\rightarrow\otimes^2(T^\ast M)
\end{equation*}

\noindent
satisfying:
\begin{equation*}
(\rho_1\otimes s_1)\oslash_Q (\rho_2\otimes s_2)= Q^{\varphi} (s_1,s_2) \rho_1 \oslash_g\rho_2\, ,
\end{equation*}
for all $\rho_1,\rho_2\in \Omega^2(M)$ and all $s_1,s_2\in \Gamma(\cS^{\varphi})$. 
\end{definition}


\subsection{Configuration space and equations of motion}


In this section, we define geometric bosonic supergravity through a
system of partial differential equations which yields a non-trivial
extension of local supergravity as described in Section
\ref{sec:localsugra}. We remark that geometric bosonic supergravity
is not expected to admit in general an action functional, which is consistent with the
fact that it implements U-duality non-trivially and therefore its
globally-defined solutions can be viewed as locally geometric supergravity
U-folds. We begin by introducing the configuration space of geometric
bosonic supergravity, which yields the space of variables of its
system of partial differential equations.

\begin{definition}
Let $\Phi$ be a scalar-electromagnetic structure on an oriented four-manifold $M$. The {\bf configuration space} of geometric bosonic supergravity on $(M,\Phi)$ is the set:
\begin{equation*}
\Conf_M(\Phi) \eqdef \left\{ (g,\varphi,\cV)  \,\, \vert \,\, g\in \mathrm{Lor}(M) \, , \,\, \varphi\in C^{\infty}(M,\cM)\, , \,\, \cV  \in \Omega^2_{+}(M,\cS^{\varphi}) \right\}\, ,
\end{equation*}
	
\noindent
where $\mathrm{Lor}(M)$ denotes the space of Lorentzian metrics on $M$.
\end{definition}

\begin{definition}
Let $\Phi$ be a scalar-electromagnetic structure on $M$. The geometric bosonic supergravity on $M$ associated to $\Phi$ is defined by the following system of partial differential equations:

\

\begin{itemize}
	\item The Einstein equations:
    \begin{equation}
    \label{eq:GlobalEinstein}
    \mathrm{Ric}^g - \frac{g}{2} R^g =\frac{g}{2}\, \mathrm{Tr}_g (\cG^{\varphi}) - \cG^{\varphi} + 2\cV\oslash_Q \cV\, .
    \end{equation}
	
	\
	
	\item The scalar equations:
	\begin{equation}
	\label{eq:GlobalScalar}
	\Tr_g(\nabla \dd\varphi) = \frac{1}{2} (\ast \cV , \Psi^{\varphi}\cV)_{g,Q^{\varphi}}\, .
	\end{equation}
	
	\noindent
	where $\nabla$ denotes the connection on $T^{\ast}M\otimes T\cM^{\varphi}$ defined as the tensor product of the Levi-Civita connection on $(M,g)$ and the pull-back by $\varphi$ of the Levi-Civita connection on $(\cM,\cG)$.
	
	\
	
	\item The Maxwell equations:
	\begin{equation}
	\label{eq:GlobalMaxwell}
	\dd_{\cD^{\varphi}} \cV = 0\, ,
	\end{equation}
\end{itemize}

\

\noindent
for triples $\Phi = (g,\varphi,\cV) \in \Conf_M(\Phi)$.
\end{definition}

\begin{remark} The configuration space $\Conf_M(\Phi) =
\Conf_M(\cG,\Delta,\cJ)$ of geometric bosonic supergravity contains as
variables the \emph{field strength} two-form instead of the
appropriate notion of gauge potential, which should be described
globally by an adequate notion of connection.  To identify the
geometrically correct notion of gauge potential we have to first
\emph{Dirac quantize} the theory, similarly to what is done with
standard Maxwell theory. In the latter theory, assuming that the field
strength has integral periods allows one to identify the gauge
potential as a connection on a certain principal $S^{1}$ bundle. The
complete Dirac quantization of four-dimensional supergravity and its
geometric interpretation has not been developed in the literature and
is currently work in progress \cite{LazaroiuShahbazi}.
\end{remark}

\begin{remark} The fact that geometric bosonic supergravity reduces
  locally to the standard formulation of local bosonic supergravity was
  proved in \cite{Lazaroiu:2016iav}, to which we refer the reader
for further details.
\end{remark}

\noindent In the following we will denote by $\mathrm{Sol}_M(\Phi) =
\Sol_M(\cG,\Delta,\cJ) \subset \Conf_M(\Phi)$ the solution set of the
geometric bosonic supergravity on $M$ associated to the
scalar-electromagnetic structure $\Phi$.


\section{The global duality group}
\label{sec:globalautgroup}


In this section we characterize the \emph{global} duality group of
geometric bosonic supergravity for a fixed scalar electromagnetic
structure $\Phi = (\cM,\cG,\Theta)$, which corresponds to the global
counterpart of the electromagnetic U-duality group of the local
theory, as discussed in Section \ref{sec:globalsymlocal}. Given a
duality structure $\Delta = (\cS,\omega,\cD)$, we denote by
$\Aut(\cS)$ the group of \emph{unbased} automorphisms of the vector
bundle $\cS\in \Delta$. Given $u\in \Aut(\cS)$ we will denote by $f_u
\colon \cM \to \cM$ the unique diffeomorphism covered by
$u$. Moreover, we denote by $\Aut(\Delta)$ the group of unbased
automorphisms of $\cS$ preserving both $\omega$ and $D$, that is:
\begin{equation*}
\Aut(\Delta) \eqdef \left\{ u\in \Aut(\cS)\,\,\vert\,\, \omega^u = \omega\, , \,\, \cD^u = \cD \right\}\, .
\end{equation*}

\noindent Given a duality structure $\Delta$ over $\cM$, the group
$\Aut(\Delta)$ has a natural left-action on $\mathrm{Lor}(M)\times
C^{\infty}(M,\cM)\times \Omega^2(\cS)$, given by:
\begin{eqnarray*}
& \mathbb{A}\colon \Aut(\Delta)\times  \mathrm{Lor}(M)\times C^{\infty}(M,\cM)\times \Omega^2(M,\cS)\to \mathrm{Lor}(M)\times C^{\infty}(M,\cM)\times \Omega^2(M,\cS)\, ,\\
& (u,g,\varphi,\cV) \mapsto (g, f_u\circ\varphi, u\cdot\cV)\, , 
\end{eqnarray*}

\noindent which gives the global counterpart of
\eqref{eq:Aaction}. For every $u\in \Aut(\Delta)$, we define:
\begin{equation*}
\mathbb{A}_{u} \colon  \mathrm{Lor}(M)\times C^{\infty}(M,\cM)\times \Omega^2(M,\cS)\to \mathrm{Lor}(M)\times C^{\infty}(M,\cM)\times \Omega^2(M,\cS)\, , \quad (g,\varphi,\cV) \mapsto (g, f_u\circ\varphi, u\cdot\cV)\, . 
\end{equation*}

\noindent This action does not preserve the configuration space
$\Conf_U(\cG,\cJ)$ of a given scalar-electromagnetic structure $\Phi =
(\cG,\cJ)$. Instead, we have the following result, which gives the
global counterpart of Theorem \ref{thm:equivsolutions}.

\begin{thm}\cite[Theorem 3.15]{Lazaroiu:2016iav}
\label{thm:equivsolutionsglobal}
For every $u\in \Aut(\Delta)$, the map $\mathbb{A}_{u}$ defines by restriction a bijection:
\begin{equation*}
\mathbb{A}_{u}\colon \Conf_M(\cG,\Delta,\cJ)\to \Conf_M(f_{u \ast}\cG,\Delta,\cJ_u)\, , 
\end{equation*}
	
\noindent
which induces a bijection between the corresponding spaces of solutions:
\begin{equation*}
\mathbb{A}_{u}\colon \Sol_M(\cG,\Delta,\cJ)\to \Sol_M(f_{u \ast}\cG,\Delta,\cJ_u)\, ,
\end{equation*}
	
\noindent
where $f_{u \ast}\cG$ is the push-forward of $\cG$ by $f_u\colon \cM \to \cM$ and $\cJ_u$ is the bundle push-forward of $\cJ$ by $u$.
\end{thm}

\begin{remark} Since elements in $\Aut(\cS)$ may cover non-trivial
diffeomorphisms of $\cM$, the pull-back/push-forward conditions appearing above must
be dealt with care. More explicitly, define the following action of
$\Aut(\cS)$ on sections of $\cS$:
\begin{equation*}
u\cdot s = u\circ s \circ f_u^{-1} \colon M \to \cS \, , \qquad u\in \Aut(\cS)\, , \qquad s\in \Gamma(\cS)\, .
\end{equation*}
	
\noindent
This action defines an isomorphism of real vector spaces $u\colon \Gamma(\cS) \to \Gamma(\cS)$ for every element $u\in \Aut(\cS)$. We have $\omega^u = \omega$ if and only if:
\begin{equation*}
(\omega^u)(s_1,s_2) \eqdef \omega(u\cdot s_1 , u\cdot s_2) \circ f_u = \omega(s_1,s_2)\, , \quad \forall\,\, s_1 , s_2 \in \Gamma(\cS)\, .
\end{equation*}
	
\noindent
Likewise, $\cD^u = \cD$ if and only if:
\begin{equation*}
\cD^u_{X}(s) \eqdef  u^{-1}\cdot (\cD_{f^{\ast}_u X} (u\cdot s)) = \cD_X (s)\, , \qquad \forall\,\, s\in \Gamma(\cS)\, , \qquad \forall\,\, X\in \mathfrak{X}(\cM)\, ,
\end{equation*}

\noindent
where $f^{\ast}_u X\in \mathfrak{X}(\cM)$ is the pull-back of $X\in\mathfrak{X}(\cM)$ by $f_u\colon \cM\to \cM$. Moreover, the explicit push-forward of $\cJ$ by $u\in \Aut(\cS)$ is given as follows:
\begin{equation*}
\cJ_u(s) \eqdef u\cdot(\cJ(u^{-1}\cdot s)) = u\circ\cJ(u^{-1}\circ s)\, ,  
\end{equation*}

\noindent
for every $s\in \Gamma(S)$.
\end{remark}

\noindent Therefore, $\Aut(\Delta)$ yields the global counterpart of
the \emph{pseudo-duality group} considered in \cite{Hull:1995gk},
which is given by $\Diff(\cM)\times \Sp(2n_v,\mathbb{R})$, and therefore 
differs remarkably from the latter if $\Delta$ is non-trivial. 
Every $u\in \Aut(\Delta)$ maps the configuration and solutions 
spaces of the supergravities associated to $(\cG,\Delta,\cJ)$ to those associated to 
$(f_{u\ast}\cG,\Delta,\cJ_u)$. Denote by $\Aut_b(\Delta) \subset \Aut(\Delta)$ 
the subgroup consisting of automorphisms of $\Aut(\Delta)$ covering the identity. 
We have the short exact sequence:
\begin{equation*}
1 \to \Aut_b(\Delta) \to \Aut(\Delta) \to \Diff_{\Delta}(\cM) \to 1\, ,
\end{equation*}

\noindent
where $\Diff_{\Delta}(\cM)$ is the subgroup of the
orientation-preserving diffeomorphism group of $\cM$ that can be
covered by elements in $\Aut(\Delta)$, which necessarily contains the
identity component of $\Diff(\cM)$. The proof of the following important lemma can be found in \cite{Donaldson}.

\begin{lemma}
\label{lemma:Holfinite}
Let $\Delta$ be a duality structure and $m \in \cM$. We have a canonical isomorphism:
\begin{equation*}
\Aut_b(\Delta) = \mathrm{C}(\mathrm{Hol}_m(\cD), \Aut(S_m,\omega_m))\, ,
\end{equation*}
	
\noindent where $\mathrm{Hol}_m(\cD)$ denotes the holonomy group of
$\cD$ at $m \in \cM$, $\Aut(S_m,\omega_m) \simeq
\mathrm{Sp}(2n_v,\mathbb{R})$ is the automorphism group of the fiber
$(S_m, \omega_m) =(S,\omega)\vert_m$ and
$\mathrm{C}(\mathrm{Hol}_m(\cD), \Aut(S_m,\omega_m)))$ denotes the
centralizer of $\mathrm{Hol}_m(\cD)$ in $\Aut(S_m,\omega_m)$. In
particular, $\Aut_b(\Delta)$ is finite-dimensional.
\end{lemma}
 
\noindent We now introduce a global counterpart of the local electromagnetic U-duality group
which is traditionally studied in the supergravity literature and was
discussed in Section \ref{sec:localsugra}.

\begin{definition} Let $\Phi = (\cG,\Delta,\cJ)$ be a
scalar-electromagnetic structure on $\cM$. We define the {\bf
electromagnetic U-duality group} $\mathfrak{S}(\Phi)$ of $\Phi$, or
{\bf U-duality group} for short, as the subgroup of $\Aut(\Delta)$
which preserves both the metric $\cG$ and the taming $\cJ$. That is:
\begin{equation*}
\mathfrak{S}(\Phi) \eqdef \left\{ u\in \Aut(\Delta)\,\,\vert\,\, f_{u\ast}\cG=\cG \, , \, \, \cJ^{u} = \cJ \right\}\, ,
\end{equation*} 
	
\noindent
where $\cJ \in \Phi$.
\end{definition}

\begin{remark}
We have $\cJ^{u} = \cJ$ if and only if:
\begin{equation*}
\cJ(u\circ s) = u\circ \cJ(s)\, , \qquad \forall\,\, s\in \Gamma(s)\, , 
\end{equation*}
\noindent
where $\circ$ composition of maps.
\end{remark}

\noindent We denote by $\Aut_b(\Theta) \subset \Aut_b(\Delta) $ the
based automorphisms of $\Delta$, which are vector bundle isomorphisms
covering the identity and preserving both $\Delta$ and $\cJ$. The
U-duality group $\mathfrak{G}$ fits into the following short exact
sequence:
\begin{equation*}
1 \to \Aut_b(\Theta)\to \mathfrak{S}(\Phi)\to \Iso_{\Phi}(\cM,\cG) \to 1\, ,
\end{equation*}
 
\noindent where $\Iso_{\Phi}(\cM,\cG)\subset \Iso(\cM,\cG)$ is the
subgroup of the isometry group of $(\cM,\cG)$ that can be covered by
elements in $\mathfrak{S}(\Phi)$. Since $\Aut_b(\Theta)\subset
\Aut_b(\Delta)$, Lemma \ref{lemma:Holfinite} implies that
$\Aut_b(\Theta)$ is finite-dimensional. Moreover,
$\Iso_{\Phi}(\cM,\cG)$ is well-known to be a finite-dimensional
Lie group, which in turn implies that $\mathfrak{S}(\Phi)$ is
a finite-dimensional Lie group which yields the global counterpart of
the local electromagnetic U-duality group defined in
\eqref{eq:localUduality}. The U-duality group of a
supergravity theory maps solutions of that theory to solutions and thus it can
be used as a \emph{solution generating} mechanism, as the following
corollary of Theorem \ref{thm:equivsolutionsglobal} states.
 
\begin{cor} The U-duality group $\mathfrak{S}(\Phi)$ of the
supergravity theory associated to $\Phi$ preserves $\Sol_M(\Phi)$, that
is, it maps solutions to solutions. In particular, every $u\in
\mathfrak{S}_M(\Phi)$ defines a bijection from $\Sol_M(\Phi)$ to itself.
\end{cor}


\subsection{Holonomy trivial duality structure}
\label{sec:holonomytrivial}


In this section we consider the U-duality group in the special case
when the duality structure $\Delta$ is holonomy trivial, that is, when
it admits a global flat symplectic frame. Fixing such a frame $\cE =
(e_1 , \hdots , e_{n_v} , f_1 , \hdots , f_{n_v})$, whose dual coframe
we denote by $\cE^{\ast} = (e^{\ast}_1 , \hdots , e^{\ast}_{n_v} ,
f^{\ast}_1 , \hdots , f^{\ast}_{n_v})$, we can canonically identify
$\Delta$ as follows:
\begin{equation*}
\cS=  \cM\times \mathbb{R}^{2n_v}\, , \qquad \omega = \sum_j f^{\ast}_j \wedge e^{\ast}_j\, , \qquad \cD = \dd \colon \Omega(\cM,\mathbb{R}^{2n_v}) \to \Omega(\cM,\mathbb{R}^{2n_v})\, ,
\end{equation*}

\noindent where $\dd$ denotes the standard exterior derivative acting
on forms taking values on $\mathbb{R}^{2n_v}$. A taming $\cJ\in
\Aut(\cS)$ of $\Delta$ is equivalent through this identification to a
unique smooth taming map:
\begin{equation*} \cJ\colon \cM\to \Aut(\mathbb{R}^{2n_v})\, .
\end{equation*}

\noindent Moreover, $\cE$ yields a canonical identification of the
unbased automorphism group of $\Aut(\cS)$:
\begin{equation*}
\Aut(S) = \Diff(\cM)\times C^{\infty}(\cM,\Aut(\mathbb{R}^{2n_v}))\, , 
\end{equation*}

\noindent
whose action is given by:
\begin{equation*}
(f,\mathfrak{U})(p,v) = (f(p) , \mathfrak{A}(p)(v))\, ,
\end{equation*}

\noindent for every $(p,v)\in \cM\times \mathbb{R}^{2n_v}$ and
$(f,\mathfrak{A}) \in \Aut(S) = \Diff(\cM)\times
C^{\infty}(\cM,\Aut(\mathbb{R}^{2n_v}))$. An element $(f,\mathfrak{A})
\in \Aut(S) $ preserves $\omega$ and $\cD$ if and only if
$\mathfrak{A}$ is constant and belongs to the symplectic group
$\Sp(2n_v,\mathbb{R})\subset \Aut(\mathbb{R}^{2n_v})$ defined as
the stabilizer of $\omega$ in $\Aut(\mathbb{R}^{2n_v})$. Therefore:
\begin{equation*}
\Aut(\Delta) = \Diff(\cM) \times \Sp(2n_v , \mathbb{R})\, ,
\end{equation*}

\noindent which corresponds to the group $\Diff(V)\times
\Sp(2n_v,\mathbb{R})$ considered in section \ref{sec:localsugra}. The
pullback $\cJ^u$ of $\cJ$ by $u = (f,\mathfrak{A})\in \Aut(\Delta)$
reads:
\begin{equation*}
\cJ^{u^{-1}}\vert_p (p,v) = (u\cdot \cJ)\vert_{f^{-1}(p)}(u^{-1}\vert_p\cdot (p,v)) = u\cdot \cJ\vert_{f^{-1}(p)}(f^{-1}(p), \mathfrak{A}^{-1}(v)) = (p, \mathfrak{A}\cJ_{f^{-1}(p)}\mathfrak{A}^{-1}(v))
\end{equation*}

\noindent
Thus an element $(f,\mathfrak{A}) \in \Aut(\Delta)$ preserves $\cJ\colon \cM\to \Aut(\mathbb{R}^{2n_v})$ if and only if:
\begin{equation*}
\mathfrak{A} (\cJ\circ f^{-1}) \mathfrak{A}^{-1} = \cJ\, ,
\end{equation*}

\noindent
which in turn implies that the electromagnetic U-duality group associated to a scalar electromagnetic structure $\Phi = (\cG,\Delta,\cJ)$ is given by:
\begin{equation}
\mathfrak{S}(\Phi) \eqdef \left\{ (f,\mathfrak{A})\in  \Aut(\Delta)\,\, \vert \,\, f_{\ast}\cG = \cG \, , \,\, \mathfrak{A}\, \cJ\, \mathfrak{A}^{-1} = \cJ \circ f \right\}\, ,
\end{equation}

\noindent
which recovers equation \eqref{eq:localUduality}.


\section{Supergravity Killing spinor equations}
\label{sec:KSE}


In the previous sections we discussed the generic bosonic sector of
four-dimensional supergravity, which a priori does not involve
supersymmetry and relies only on a consistent coupling of gravity,
scalars and abelian gauge fields in a manner compatible with electromagnetic
duality. In order to have the complete picture of geometric
supergravity and to showcase the power of supersymmetry, we need to discuss
the \emph{supergravity} Killing spinor equations, which arise from
imposing invariance under supersymmetry transformations on a purely
bosonic solution. For the moment, we denote by $\cB$ the bosonic
fields of a four-dimensional supergravity theory, which we know from Section
\ref{sec:geometricsugra} to consist of Lorentzian metrics, smooth
maps into the scalar manifold of the theory and twisted self-dual
two-forms taking values in the duality bundle, and let us denote by
$\cF$ the corresponding \emph{fermionic fields}. The latter depend
heavily on the specific supergravity theory under consideration. Given
a \emph{supersymmetry parameter} $\varepsilon$, which we can think of
as being a spinor on $M$\footnote{More precisely, it is a section of a
bundle of real or complex Clifford modules of certain type, which is
in general not associated to a spin structure but the more general
notion of \emph{Lipschitz structure} instead, see
\cite{Lazaroiu:2016vov,Lazaroiu:2016nbq} for more details.}, the
\emph{infinitesimal} supersymmetry transformations of $\cB$ and $\cF$
in the direction $\varepsilon$ correspond schematically to an infinitesimal
transformation of the form:
\begin{equation*}
\delta_{\varepsilon} \cB= \cF(\varepsilon)\, , \qquad \delta_{\varepsilon} \cF= \cB(\varepsilon)
\end{equation*}

\noindent where the right hand side depends linearly on
$\varepsilon$. A solution $(\cB,\cF)$ of four-dimensional supergravity
is said to be \emph{supersymmetric} if it is invariant under such an infinitesimal transformation, i.e.:
\begin{equation}
\label{eq:susysch1}
\delta_{\varepsilon} \cB= \cF(\varepsilon) = 0\, , \qquad \delta_{\varepsilon} \cF= \cB(\varepsilon) = 0\, .
\end{equation}

\noindent To the best of our knowledge, there is no fully general and
mathematically rigorous formulation of these transformations which
could serve to give the basis of a mathematical theory of supergravity
including its complete fermionic sector and supersymmetry
transformations. As disappointing as this may seem, what is important
to us is that if we restrict the previous transformations to a purely
bosonic background, that is, if we set $\cF=0$, \eqref{eq:susysch1}
reduces to an expression of the form:
\begin{equation}
\label{eq:susysch2}
\delta_{\varepsilon} \cF= \cB(\varepsilon) = 0\, ,
\end{equation}

\noindent which is expected to admit a rigorous mathematical
formulation using the tools of mathematical gauge theory and global
differential geometry and analysis. In the case of four-dimensional ungauged supergravity, equation \eqref{eq:susysch2} yields a system of partial differential equations for a metric $g$, a
scalar map $\varphi$ and a twisted self-dual two-form $\cV$ coupled to
a spinor $\varepsilon$. In agreement with the terminology introduced
earlier, a bosonic solution $\cB$ is then said to be supersymmetric if
Equation \eqref{eq:susysch2} holds. The spinorial equations arising
from $\delta_{\varepsilon} \cF = 0$ are always of the type:
\begin{equation*}
\mathfrak{D}_{\cB}\varepsilon = 0\, , \qquad \mathfrak{Q}_{\cB}(\varepsilon) = 0\, ,
\end{equation*}

\noindent where $\varepsilon\in \Gamma(S)$ is a section of an
appropriate bundle of real or complex Clifford modules over the
underlying manifold $M$, $\mathfrak{D}_{\cB}$ is a connection on $S$
depending on $\cB$ and $\mathfrak{Q}_{\cB}\in \Gamma(End(S))$ is an
endomorphism of $S$ depending also on $\cB$. We note that the
mathematical theory of supergravity Killing spinor equations is far
from being established, so in the following we will content ourselves
with presenting some particular examples where such mathematical
formulation does exist, see \cite{Cortes:2018lan} for more
details. The main difficulty in developing the mathematical theory of
supergravity Killing spinor equations resides in giving global
mathematical sense to the local formulas available in the supergravity
physics literature for $\mathfrak{D}_{\cB}$ and
$\mathfrak{Q}_{\cB}(\varepsilon)$, which involve state of the art
geometric structures subtly coupled through supersymmetry.
 

\subsection{Pure (AdS) $\cN=1$ supergravity}


We fix an oriented Lorentzian spin four-manifold, which for simplicity
in the exposition we will assume to satisfy $H^1(M,\mathbb{Z}_2) = 0$
(so the spin structure is unique up to isomorphism). For
every Lorentzian metric $g$ on $M$, we denote by $S_g$ the unique
(modulo isomorphism) bundle of irreducible real Clifford modules over
the bundle of Clifford algebras $\Cl(M,g)$ of $(M,g)$. Pure (AdS)
$\cN=1$ supergravity is the simplest four-dimensional supergravity
theory. The scalar manifold consists of a point and the duality
structure is trivial of zero rank. The scalar potential is
constant. The bosonic \emph{matter content} of the theory, that is,
its configuration space, consists therefore simply of a Lorentzian
metric and the theory admits the following action functional
\cite[Chapter 5]{Ortin}:
\begin{equation*}
\mathfrak{S}[g]=\int_U \left[\mathrm{R}_{g} + 6 \lambda^2\right] \vol_g\, .
\end{equation*}

\noindent
The partial differential equations associated to the variational problem of the previous functional are:
\begin{equation*}
\mathrm{Ric}(g) =  - 3 \lambda^2\, g\, .
\end{equation*}

\noindent Therefore, bosonic pure AdS $\cN=1$ supergravity is given by
Einstein's theory of gravity coupled to a non positive cosmological
constant. The Killing spinor equations of the theory read
\cite[Chapter 5]{Ortin}:
\begin{equation}
\label{eq:KSEPureAdS}
\nabla^g_v \varepsilon = \frac{\lambda}{2}v\cdot \varepsilon\, , \qquad \forall\,\, v\in \mathfrak{X}(M)\, ,
\end{equation}

\noindent for a real spinor $\varepsilon\in
\Gamma(S_g)$. Consequently, an Einstein metric $g$ with Einstein
constant $-3\lambda^2$ is a supersymmetric solution of $\cN=1$ pure
AdS supergravity if and only if $(M,g)$ admits a spinor $\varepsilon$
satisfying \eqref{eq:KSEPureAdS}. Hence, the set of supersymmetric solutions of pure (AdS)
$\cN=1$ supergravity is given by:
\begin{equation*}
\mathrm{Sol}_{S}(M,\lambda) = \left\{ (g,\epsilon)\,\, \vert \,\, \mathrm{Ric}(g) =  - 3 \lambda^2\, g\, , \,\, \,\, \nabla^g_v \varepsilon = \frac{\lambda}{2}v\cdot \varepsilon\, , \,\, \forall v\in \mathfrak{X}(M) \right\}\, .
\end{equation*}

\noindent Equation \eqref{eq:KSEPureAdS} is a particular case of a
\emph{real} Killing spinor equation on a Lorentzian four-manifold, and
has been studied in \cite{Leitner,Cortes:2019xmk}. Reference
\cite{Leitner} proves that an oriented and spin Lorentzian
four-manifold carrying a solution of \eqref{eq:KSEPureAdS} such that
$\lambda \neq 0$ is locally conformally a Brinkmann space-time. On the
other hand, Reference \cite{Cortes:2019xmk} proves the following
global result.

\begin{thm}\cite[Theorem 5.3]{Cortes:2019xmk}
\label{thm:RealKSE} $(M,g)$ admits a nontrivial real Killing spinor
with Killing constant $\frac{\lambda}{2}$ if and only if it admits a
pair of orthogonal one-forms $u,l\in \Omega^1(M)$ with $u$ lightlike
and $l$ of positive unit norm satisfying:
\begin{equation*}
\nabla^g u = \lambda\, u\wedge l\, , \qquad  \nabla^g l = \kappa\otimes u + \lambda(l\otimes l - g)\, ,
\end{equation*}

\noindent for some $\kappa\in \Omega^1(M)$. In this case,
$u^{\sharp}\in \mathfrak{X}(M)$ is a Killing vector field with
geodesic integral curves.
\end{thm}

\begin{remark} Theorem \ref{thm:RealKSE} immediately implies that
$\kappa$ is closed if and only if Leitner's result holds with respect
to $u$, that is, if and only if every such $(M,g)$ is locally
conformally Brinkmann with respect to $u$. We have not been able to
prove that $\kappa$ is necessarily closed.
\end{remark}
 
\noindent Of course, when $\lambda = 0$ equation \eqref{eq:KSEPureAdS}
reduces to the condition of $\varepsilon$ being a parallel spinor,
which has been extensively studied both in the mathematics and physics
literature, see for example \cite{Leistner} and references therein. To
the best knowledge of the authors, the differential topology of
globally hyperbolic Lorentzian manifolds carrying a solution of
\eqref{eq:KSEPureAdS} has not been investigated in the literature. We
believe that the global characterization provided by Theorem
\ref{thm:RealKSE} is a convenient starting point for such a study.


\subsection{Chiral $\cN=1$ supergravity with constant scalar map and superpotential}


We fix an oriented and spin Lorentzian four-manifold, which we 
assume to satisfy $H^1(M,\mathbb{Z}_2) = 0$ for the same reasons as in
the previous section. For every Lorentzian metric $g$ on $M$, we
denote by $\$_g$ the unique (modulo isomorphism) bundle of irreducible
complex Clifford modules over $\Cl(M,g)$. The Lorentzian volume form
of $(M,g)$ is denoted by $\nu$, while the complex volume form is
denote by $\nu_{\C} = i \nu$. We have:
\begin{eqnarray}
\nu_{\C}^2 = 1\, , 
\end{eqnarray}

\noindent
Therefore, the complex spinor bundle splits as a sum of chiral bundles:
\begin{equation*}
\$_g = \$^{-}_g \oplus \$^{+}_g\, ,
\end{equation*}

\noindent
where the superscript denotes the chirality. To describe chiral
$\cN=1$ supergravity with constant scalar map and superpotential (see
\cite{FreedmanProeyen,Ortin} for more details) we take the scalar
manifold to be a point and the duality structure to be trivial of rank
zero. Under these assumptions, it can be shown that the theory admits
an action functional given by:
\begin{equation*}
\mathfrak{S}[g]=\int_U \left[\mathrm{R}_{g} + 24 \vert w\vert^2\right] \vol_g\, ,
\end{equation*}

\noindent where $w\in \mathbb{C}$ is a complex constant corresponding
to the superpotential and $U\subset M$ is any relatively compact open
set. The equations of motion associated to the previous functional
read:
\begin{equation*}
\mathrm{Ric}(g) = - 12 \vert w\vert^2 g \, ,
\end{equation*}

\noindent which are the standard Einstein equations coupled to a
non positive cosmological constant. Using the chiral
splitting $\$_g = \$^{-}_g \oplus \$^{+}_g$, we define for every $w\in
\mathbb{C}$ the following morphism of complex vector bundles:
\begin{equation*}
T_w\colon \Omega^{0}(\$_g) \to \Omega^1(\$_g) \, , \qquad T_w(\epsilon_1\oplus \epsilon_{2})(v) = \gamma(v)(w \epsilon_1\oplus \bar{w} \epsilon_{2})\, ,
\end{equation*}

\noindent
where $\gamma\colon \Cl(M,g)\to \End(\$_g)$ denotes Clifford
multiplication. In terms of $T_w\colon \Omega^{0}(\$_g) \to
\Omega^1(\$_g)$ the Killing spinor equations of the theory are given
by:
\begin{equation}
\label{eq:complexKSE1}
\nabla^{g}\varepsilon = T_{w}(\varepsilon)\, , \qquad \mathfrak{c}(\varepsilon) = \varepsilon\, ,
\end{equation}

\noindent
where $\nabla^g$ denotes the lift of the Levi-Civita connection to the spinor bundle and:
\begin{equation*}
\mathfrak{c}\colon \$_g \to \$_g\, 
\end{equation*}

\noindent denotes the canonical complex-conjugate and spin-equivariant
automorphism of the complex spinor bundle $\$_g$ (see
\cite{Cortes:2018lan} for more details). Therefore, the set
$\mathrm{Sol}_{S}(M)$ of supersymmetric solutions on $M$ consists on
pairs $(g,\epsilon)$, with $g$ a Lorentzian metric and $\epsilon$
chiral spinor, such that:
\begin{equation*}
\mathrm{Sol}_{S}(M) = \left\{ (g,\epsilon)\,\, \vert \,\, \mathrm{Ric}(g) =  - 12 \vert w\vert^2\, g\, , \,\, \,\, \nabla^g \epsilon = T_w(\epsilon) \, , \,\,\,\, \mathfrak{c}(\varepsilon) = \varepsilon\, , \,\, \forall v\in \mathfrak{X}(M) \right\}\,\, .
\end{equation*}

\noindent
It is important to point out that Equation \eqref{eq:complexKSE1} does
not correspond to a \emph{standard} Killing spinor equation (for
neither real nor imaginary Killing spinors) unless $w$ is real, due to
the fact that the endomorphism $T_w$ involves the complex conjugate of
$w$. This in turn implies that the number that occurs as the Einstein constant of the corresponding
integrability condition is actually $|w|^2$. This allows $w$ to be any complex
number instead of only real or purely imaginary. To the best of our
knowledge, the globally hyperbolic Lorentzian four-manifolds that
admit supersymmetric solutions to this supergravity theory have not
been investigated in the literature.


\subsection{Chiral $\cN=1$ supergravity with vanishing superpotential}


We fix an oriented Lorentzian $\Spin^c(3,1)$ four-manifold, which for
simplicity of exposition we will assume to satisfy
$H^2(M,\mathbb{Z}) = 0$ (so that isomorphism classes of spin$^c$ structures
on $M$ are unique). For every Lorentzian metric $g$ on $M$, we denote
by $\$_g$ the unique (modulo isomorphism) bundle of irreducible
complex Clifford modules over $\Cl(M,g)$.  As before, the complex
spinor bundle splits as a sum of chiral bundles:
\begin{equation*}
\$_g = \$^{-}_g \oplus \$^{+}_g\, ,
\end{equation*}

\noindent where the superscript denotes chirality. The scalar manifold
of $\cN=1$ supergravity with vanishing superpotential and with trivial
duality structure of rank zero is a complex manifold $\cM$ equipped
with a negative Hermitian holomorphic line bundle $(\cL,\cH)$ with
Hermitian structure $\cH$. The Riemannian metric $\cG$ occurring in
the non-linear sigma model of four-dimensional supergravity is given
by the metric induced by the curvature of the Chern connection of
$(\cL,\cH)$, see \cite{Bagger:1983tt,Cortes:2018lan} for more
details. Such $\cN=1$ supergravity admits a Lagrangian formulation
with Lagrangian given by:
\begin{equation*}
\mathfrak{Lag}[g,\varphi]= \mathrm{R}_{g} - \vert\dd \varphi \vert^2_{\cG, g} \, , 
\end{equation*}

\noindent
for pairs $(g,\varphi)$ consisting of Lorentzian metrics $g$ and scalar maps $\varphi \colon M \to \cM$. As it is standard in the theory of harmonic maps (or wave maps), we consider:
\begin{equation*}
\dd\varphi \in \ \Omega^1(M,T\cM^{\varphi})\, ,
\end{equation*}

\noindent as a one-form on $M$ taking values in the pullback of
$T\cM^{\varphi}$ by $\varphi$. Therefore, the theory reduces to
Einstein gravity coupled to a non-linear sigma model with target space
given by the complex manifold $\cM$ equipped with the K\"ahler metric
defined by the curvature of the Chern connection of $(\cL,\cH)$. The
Killing spinor equations are given by:
\begin{equation*}
\nabla^\varphi \epsilon = 0\, , \qquad \dd \varphi^{0,1}\cdot\epsilon = 0\, ,
\end{equation*}
 
\noindent where $\nabla^\varphi\colon \Gamma(\$_g) \to \Gamma(\$_g)$
is the canonical lift of the Levi-Civita connection on $(M,g)$
together with the pull-back of the Chern connection
on $(\cL,\cH)$ by $\varphi$. Therefore, the set
$\mathrm{Conf}_{S}(M,\cM,\cL,\cH)$ of supersymmetric configurations on
$M$ consists on triples $(g,\varphi,\epsilon)$, with $g$ a Lorentzian
metric, $\varphi \colon M\to \cM$ a scalar map and $\epsilon$ a chiral
spinor, such that:
\begin{equation*}
\mathrm{Conf}_{S}(M,\cM,\cL,\cH) = \left\{ (g,\varphi,\epsilon)\,\, \vert \,\, \nabla^\varphi \epsilon = 0\, , \,\,  \dd \varphi^{0,1}\cdot\epsilon = 0 \right\}\, .
\end{equation*}

\noindent
Lorentzian manifolds $(M,g)$ admitting a solution
$(g,\varphi,\epsilon)$ to the Killing spinor equations stated above are
particular instances of Lorentzian $\Spin^c(3,1)$ manifolds admitting
parallel spinors. Simply connected and geodesically complete
Lorentzian manifolds admitting spin-c parallel spinors have been studied and classified in
the literature, see reference \cite{Moroianu} for the Riemannian case
and Reference \cite{Ikemakhen} for the pseudo-Riemannian case. It
cannot be expected a priori that every $\Spin^{c}(3,1)$ Lorentzian
four-manifold which admits a parallel spinor also admits a solution to
the above Killing spinor equations. Adapting the main Theorem of
\cite{Ikemakhen} to our situation we obtain the following result.

\begin{prop}
	\label{prop:candidates} Let $M$ be a simply-connected and
geodesically complete Lorentzian four-manifold admitting a
supersymmetric solution $(g,\varphi,\epsilon)$ of $\cN=1$ chiral
supergravity with vanishing superpotential. Then, one of the following
holds:
	
	\
	
	\begin{enumerate}
		\item $(M,g)$ is isometric to four-dimensional flat
Minkowski space.
		
		\
		
		\item $(M,g)$ is isometric to $(M,g) \simeq
(\R^2\times X, \eta_{1,1}\times h)$, where $\eta_{1,1}$ is the flat
two-dimensional Minkowski metric and $X$ is a Riemann surface equipped
with a K\"ahler metric $h$.
		
		\
		
		\item The holonomy group $H$ of $(M,g)$ is a subgroup
of $\SO(2)\ltimes \mathbb{R}^2 \subset \SO_0(3,1)$, where $\SO(2)\ltimes \mathbb{R}^2$ is the stabilizer of a null vector in $\mathbb{R}^4$.
	\end{enumerate}
\end{prop}

\noindent Therefore, every geodesically complete and simply connected
supersymmetric solution must be of the form described by the previous
proposition. However, the converse need not be true, since
a supersymmetric solution requires $(M,g)$ to admit a parallel spinor
with respect to the specific connection $\nabla^{\varphi}$, which is
coupled to the scalar map $\varphi$, which is in turn required to
satisfy its corresponding Killing spinor equation. To the best of our
knowledge, the problem of classifying globally hyperbolic Lorentzian
four-manifolds carrying supersymmetric solutions of this supergravity
theory is currently open.





\begin{ack}
The work of C. I. L. was supported by grant IBS-R003-S1. The work of C.S.S. is supported by the Deutsche Excellence Strategy – EXC 2121 “Quantum Universe” - 390833306.
\end{ack}


\appendix


\end{document}